\documentclass[leqno]{amsart}
\usepackage{amsmath,amsthm,amsfonts,amssymb}

\numberwithin{equation}{section}
\theoremstyle{plain}

\newtheorem{assumption}{Assumption}
\newtheorem{theorem}{Theorem}

\newtheorem{lemma}[theorem]{Lemma}

\newtheorem{proposition}[theorem]{Proposition}

\newtheorem{remark}[theorem]{Remark}

\newcommand{\Gammaset}{\mathbf{\Gamma}}
\newcommand{\Gammafunc}{\Gamma}

\begin{document}

\title[On the Hawkes Process with Different Exciting Functions]{On the Hawkes Process with Different Exciting Functions}

\author{Behzad Mehrdad}
\address
{Courant Institute of Mathematical Sciences\newline
\indent New York University\newline
\indent 251 Mercer Street\newline
\indent New York, NY-10012\newline
\indent United States of America}
\email{mehrdad@cims.nyu.edu}

\author{Lingjiong Zhu}
\address
{Department of Mathematics \newline
\indent Florida State University \newline
\indent 1017 Academic Way \newline
\indent Tallahassee, FL-32306 \newline
\indent United States of America}
\email{
zhu@math.fsu.edu}

\date{3 September 2017.}

\begin{abstract}
The Hawkes process is a simple point process, whose intensity
function depends on the entire past history and is self-exciting and has the clustering property. 
The Hawkes process is in general non-Markovian. The linear Hawkes process has 
immigration-birth representation.
Based on that, Fierro et al. recently introduced a generalized linear Hawkes model 
with different exciting functions.
In this paper, we study the convergence to equilibrium, large deviation principle, 
and moderate deviation principle for this generalized model. This model
also has connections to the multivariate linear Hawkes process.
Some applications to finance are also discussed.
\end{abstract}

\maketitle

\tableofcontents

\section{Introduction}

\subsection{Hawkes Process}

Let $N$ be a simple point process on $\mathbb{R}$ and let $\mathcal{F}^{-\infty}_{t}:=\sigma(N(C),C\in\mathcal{B}(\mathbb{R}), C\subset(-\infty,t])$ be
an increasing family of $\sigma$-algebras. Any nonnegative $\mathcal{F}^{-\infty}_{t}$-progressively measurable process $\lambda_{t}$ with
\begin{equation}
\mathbb{E}\left[N(a,b]|\mathcal{F}^{-\infty}_{a}\right]=\mathbb{E}\left[\int_{a}^{b}\lambda_{s}ds\big|\mathcal{F}^{-\infty}_{a}\right]
\end{equation}
a.s. for all intervals $(a,b]$ is called an $\mathcal{F}^{-\infty}_{t}$-intensity of $N$. We use the notation $N_{t}:=N(0,t]$ to denote the number of
points in the interval $(0,t]$. 

A Hawkes process is a simple point process $N$ admitting an $\mathcal{F}^{-\infty}_{t}$-intensity
\begin{equation}
\lambda_{t}:=\lambda\left(\int_{-\infty}^{t}h(t-s)N(ds)\right),\label{dynamics}
\end{equation}
where $\lambda(\cdot):\mathbb{R}^{+}\rightarrow\mathbb{R}^{+}$ is locally integrable, left continuous, 
$h(\cdot):\mathbb{R}^{+}\rightarrow\mathbb{R}^{+}$ and
we always assume that $\Vert h\Vert_{L^{1}}=\int_{0}^{\infty}h(t)dt<\infty$. 
In \eqref{dynamics}, $\int_{-\infty}^{t}h(t-s)N(ds)$ stands for $\int_{(-\infty,t)}h(t-s)N(ds)=\sum_{\tau<t}h(t-\tau)$, where
$\tau$ are the occurrences of the points before time $t$.

In the literature, $h(\cdot)$ and $\lambda(\cdot)$ are usually referred to
as exciting function and rate function respectively.

When $\lambda(\cdot)$ is linear, the Hawkes process is said to be linear and it is named after Hawkes \cite{Hawkes}.
The linear Hawkes process can be studied via immigration-birth representation, see e.g. Hawkes and Oakes \cite{HawkesII}.
When $\lambda(\cdot)$ is nonlinear, the Hawkes process is said to be nonlinear and the nonlinear Hawkes process
was first introduced by Br\'{e}maud and Massouli\'{e} \cite{Bremaud}.

The law of large numbers and central limit theorems for linear Hawkes processes were studied
in e.g. Bacry et al. \cite{Bacry}, and the moderate deviations were studied in Zhu \cite{ZhuMDP}.
The central limit theorem for nonlinear Hawkes processes was obtained in Zhu \cite{ZhuCLT}.
Bordenave and Torrisi \cite{Bordenave} obtained the large deviations for linear Hawkes processes
and the large deviations for nonlinear Hawkes processes were studied in Zhu \cite{ZhuI} and Zhu \cite{ZhuII}.
The limit theorems of some generalizations of the classical Hawkes processes have been studied in e.g. Karabash and Zhu \cite{Karabash}
and Zhu \cite{ZhuCIR}.

The self-exciting and clutstering properties of the Hawkes process make it ideal
to characterize the correlations in some complex systems, including finance.
Bacry et al. \cite{Bacry}, Bacry et al. \cite{BacryII} studied microstructure noise and Epps effect using the Hawkes models.
Chavez-Demoulin et al. \cite{Chavez} studied value-at-risk.
Errais et al. \cite{Errais} used Hawkes process to model the credit risk. 
Embrechts et al. \cite{Embrechts} fit the Hawkes process to financial data.
 
The Hawkes process has also been applied to many other fields, including 
seismology, see e.g. Hawkes and Adamopoulos \cite{HawkesIV}, Ogata \cite{OgataII}, 
sociology, see e.g. Crane and Sornette \cite{Crane} and Blundell et al. \cite{Blundell},
and neuroscience, see e.g. Chornoboy et al. \cite{Chornoboy}, Pernice et al. \cite{PerniceI},
Pernice et al. \cite{PerniceII}. For a survey of the Hawkes process and its applications, we refer to Liniger \cite{Liniger}
and Zhu \cite{ZhuThesis}.

\subsection{Hawkes Process with Different Exciting Functions}

In this paper, we are interested to study an extension of the linear Hawkes process proposed by Fierro et al. \cite{Fierro}.
It is based on the immigration-birth representation structure of the linear Hawkes process.
The classical Hawkes process can be constructed from a homogeneous Poisson
process (immigration) and using the same exciting function for different generations of
offspring (birth). In some fields, e.g. seismology, where main shocks produce
aftershocks with possibly different intensities, that naturally leads to the study of a Hawkes process
with different exciting functions as proposed in Fierro et al. \cite{Fierro}.

Let $(N^{n})_{n\in\mathbb{N}}$ be a sequence of non-explosive simple point processes
without common jumps so that
\begin{itemize}
\item
$N^{0}$ is an inhomogeneous Poisson process with intensity $\gamma_{0}(t)$ at time $t$.

\item
For every $n\in\mathbb{N}$, $N^{n}$ is a simple point process with 
intensity $\lambda^{n}_{t}=\int_{0}^{t}\gamma_{n}(t-s)N^{n-1}(ds)$,
where the integral $\int_{0}^{t}\gamma_{n}(t-s)N^{n-1}(ds)$ 
denotes for $\int_{(0,t)}\gamma_{n}(t-s)N^{n-1}(ds)=\sum_{\tau\in N^{n-1},0<\tau<t}\gamma_{n}(t-\tau)$, and $\gamma_{n}(\cdot):\mathbb{R}^{+}\rightarrow\mathbb{R}^{+}$. Note that by definition, the intensity is $\mathcal{F}_{t}$-predictable. 

\item
For every $n\in\mathbb{N}\cup\{0\}$, conditional on $N^{0},\ldots, N^{n}$, 
$N^{n+1}$ is a inhomogeneous Poisson process with intensity $\lambda^{n+1}$.
\end{itemize}
The existence of such a process was proved as Proposition 2.1. in Fierro et al. \cite{Fierro}.

Using the notation of immigration-birth representation, 
$N^{0}$ is called the immigrant process
and $N^{n}$ the $n$th generation offspring process.

Let $N:=\sum_{n=0}^{\infty}N^{n}$. $N$ is said to be the Hawkes process with excitation functions 
$(\gamma_{n})_{n\in\mathbb{N}\cup\{0\}}$. If $\gamma_{0}(t)\equiv\overline{\gamma}_{0}>0$ and $\gamma_{n}(t)=h(t)$ for any $n\in\mathbb{N}$,
then the model reduces to the classical linear Hawkes process $N$ with intensity at time $t$ given by
\begin{equation}\label{ClassicalLinear}
\lambda_{t}=\overline{\gamma}_{0}+\int_{0}^{t}h(t-s)N(ds).
\end{equation}

\begin{assumption}\label{MainAssumption}
(i) $\overline{\gamma}_{0}:=\lim_{t\rightarrow\infty}\frac{1}{t}\int_{0}^{t}\gamma_{0}(s)ds$ exists and is finite.

(ii) $\rho:=\sup_{n\in\mathbb{N}}\int_{0}^{\infty}\gamma_{n}(t)dt<1$.
\end{assumption}

Under Assumption \ref{MainAssumption}, Fierro et al. \cite{Fierro} showed that for any $t\geq 0$,
\begin{equation}
\mathbb{E}[N_{t}]=\int_{0}^{t}\sum_{n=0}^{\infty}(\gamma_{0}\ast\cdots\ast\gamma_{n})(s)ds<\infty.
\end{equation}

Fierro et al. \cite{Fierro} proved the following law of large numbers result under Assumption \ref{MainAssumption},
\begin{equation}\label{LLN}
\frac{N_{t}}{t}\rightarrow m,\qquad\text{almost surely, as $t\rightarrow\infty$},
\end{equation}
where $m:=\sum_{n=0}^{\infty}m_{n}$, $m_{0}:=\overline{\gamma}_{0}$ and
\begin{equation}\label{mn}
m_{n}:=\overline{\gamma}_{0}\prod_{i=1}^{n}\int_{0}^{\infty}\gamma_{i}(u)du,\qquad n\in\mathbb{N}.
\end{equation}
It is easy to check that in the case of classical linear Hawkes process \eqref{ClassicalLinear},
\begin{equation}
m=\overline{\gamma}_{0}\sum_{n=1}^{\infty}\Vert h\Vert_{L^{1}}^{n}=\frac{\overline{\gamma}_{0}}{1-\Vert h\Vert_{L^{1}}},
\end{equation}
which is consistent with the results in Hawkes \cite{Hawkes}.

\begin{assumption}\label{AssumptionTwo}
\begin{equation}
\lim_{t\rightarrow\infty}\sqrt{t}\left[\frac{1}{t}\int_{0}^{t}\gamma_{0}(s)ds-\overline{\gamma}_{0}\right]=0,
\end{equation}
and
\begin{equation}
\lim_{t\rightarrow\infty}\sqrt{t}\int_{t}^{\infty}\sum_{p=1}^{\infty}
\gamma_{p}\ast\cdots\ast\gamma_{1}(s)ds=0.
\end{equation}
\end{assumption}

Further assume Assumption \ref{AssumptionTwo}, 
Fierro et al. \cite{Fierro} also obtained the central limit theorem, 
which is the main result of their paper,
\begin{equation}
\frac{N_{t}-mt}{\sqrt{t}}\rightarrow N(0,\sigma^{2}),
\end{equation}
in distribution as $t\rightarrow\infty$, where
\begin{equation}
\sigma^{2}:=\sum_{j=0}^{\infty}\left(1+\sum_{p=1}^{\infty}\prod_{i=j+1}^{p+j}\int_{0}^{\infty}\gamma_{i}(u)du\right)^{2}m_{j}.
\end{equation}
It is easy to check that in the case of classical linear Hawkes process \eqref{ClassicalLinear},
\begin{equation}
\sigma^{2}=\sum_{j=0}^{\infty}\left(1+\sum_{p=0}^{\infty}\Vert h\Vert_{L^{1}}^{p}\right)^{2}\overline{\gamma}_{0}\Vert h\Vert_{L^{1}}^{j}
=\frac{\overline{\gamma}_{0}}{(1-\Vert h\Vert_{L^{1}})^{3}},
\end{equation}
which is consistent with the results in Bacry et al. \cite{Bacry}.

The paper is organized as the following. In Section \ref{ErgodicSection}, we show that there exists
a stationary version of the Hawkes process with different exciting functions and we will show the convergence
to the equilibrium. In Section \ref{ConnectionSection}, we will point out the connections of the Hawkes process
with different exciting functions to the classical multivariate linear Hawkes process, which has been well studied
in the literature. In Section \ref{MDPLDPSection}, we obtain both the large deviations and the moderate deviations
for the model. Finally, we discuss some applications to finance in Section \ref{FinanceSection}.

\section{Convergence to Equilibrium}\label{ErgodicSection}

Assume that $\gamma_{0}\equiv\overline{\gamma}_{0}$ is a positive constant 
and Assumption \ref{MainAssumption} (ii) holds, then, there exists a stationary version
of the Hawkes process $N^{\dagger}$ with exciting functions 
$(\gamma_{n})_{n\in\mathbb{N}}\cup\{\overline{\gamma}_{0}\}$ constructed as follows.

Let $N^{\dagger,0}$ be a homogeneous Poisson process with intesntiy $\overline{\gamma}_{0}$
on $\mathbb{R}$ and for each $n\in\mathbb{N}$, $N^{\dagger,n}$ is an inhomogeneous Poisson
process with intensity
\begin{equation}
\lambda^{\dagger,n}_{t}=\int_{-\infty}^{t}\gamma^{n}(t-s)N^{\dagger,n-1}(ds),
\end{equation}
and $N^{\dagger}=\sum_{n=0}^{\infty}N^{\dagger,n}$.

The space of integer-valued measures is endowed with the vague topology, i.e.
$N^{n}$ converges to $N$ if and only if for any continuous function $\phi$ with compact support,
$\int\phi(x)N^{n}(dx)\rightarrow\int\phi(x)N(dx)$.

Given a simple point process $N$ on $\mathbb{R}$, one can define $\theta_{t}N$ as the
process shifted by time $t$, i.e. $\theta_{t}N(A)=N(A+t)$, where $A+t:=\{s+t:s\in A\}$
for any Borel set $A$ associated with the vague topology. 

We say a sequence of simple point processes $N^{n}$ converges to 
a simple point process $N$ in distribution if for any Borel set $A$ associated
with the vague topology, $\lim_{n\rightarrow\infty}\mathbb{P}(N^{n}\in A)=\mathbb{P}(N\in A)$
and the convergences is in variation if 
\begin{equation}
\lim_{n\rightarrow\infty}\sup_{A}|\mathbb{P}(N^{n}\in A)-\mathbb{P}(N\in A)|=0.
\end{equation}
This is the notation given in Br\'{e}maud and Massouli\'{e} \cite{Bremaud}.

In Daley and Vere-Jones \cite{Daley}'s terminology, convergence in distribution (variation)
is referred to as the weak (strong) convergence and the stationarity associated
with the stationary limit is referred to as the weak (strong) stationarity.

For a given simple point process $N$ on $\mathbb{R}$, 
let $N^{+}$ be its restriction to $\mathbb{R}^{+}$.

\begin{theorem}\label{ergodicity}
Let $N=\sum_{n=0}^{\infty}N^{n}$ be the Hawkes process with exciting functions 
$(\gamma_{n})_{n\in\mathbb{N}}\cup\{\overline{\gamma}_{0}\}$ with empty history, i.e.
$N(-\infty,0]=0$ and satisfies Assumption \ref{MainAssumption} (ii). Then, the following is true.

(i) $\theta_{s}N$ converges to $N^{\dagger}$ weakly
as $s\rightarrow\infty$, i.e. $(\theta_{s}N)^{+}$ converges in distribution
to $(N^{\dagger})^{+}$.

(ii) If we further assume that $\int_{0}^{\infty}t\gamma_{1}(t)dt<\infty$,
then, $\theta_{s}N$ converges to $N^{\dagger}$ strongly
as $s\rightarrow\infty$, i.e. $(\theta_{s}N)^{+}$ converges in variation
to $(N^{\dagger})^{+}$.
\end{theorem}

\begin{proof}
(i) For both $N^{\dagger}$ and $N$, let $\theta_{s}N^{\dagger}$ and $\theta_{s}N$
be the shifted version obtained by setting time $s$ as the origin and shift the process
backwards in times by $s$ to bring the origin back to $0$, 
that is, $\theta_{s}N(A)=N(A+s)$, where $A+s:=\{t+s:t\in A\}$ for any Borel set $A$.
We can decompose $\theta_{s}N^{\dagger}$ into two components, one component
has the same dynamics as $\theta_{s}N$, being built from the points generated by
the homogeneous Poisson process $\overline{\gamma}_{0}$ and its offspring $(N^{n})_{n\geq 1}$
after time $-s$, the other component $N^{\dagger}_{-s}$ that consists of the offspring
of the points generated by homogeneous Poisson process $\overline{\gamma}_{0}$ before time $-s$.
Hence, we have
\begin{equation}
\lambda_{-s}^{\dagger,1}(t)=\int_{-\infty}^{-s}\gamma_{1}(t-u)N^{\dagger,0}_{-s}(du),
\qquad t\geq -s,
\end{equation}
and 
\begin{equation}
\lambda_{-s}^{\dagger,n}(t)=\int_{-\infty}^{t}\gamma_{n}(t-u)N^{\dagger,n-1}_{-s}(du),
\qquad t\geq -s, n\geq 2.
\end{equation}
Let us define
\begin{equation}
H_{n}(t):=\int_{t}^{\infty}\gamma_{n}(s)ds,\qquad n\geq 1.
\end{equation}
It is easy to compute that
\begin{equation}
\mathbb{E}[\lambda_{-s}^{\dagger,1}(t)]=\int_{-\infty}^{-s}\gamma_{1}(t-u)\overline{\gamma}_{0}du
=\overline{\gamma}_{0}\int_{t+s}^{\infty}\gamma_{1}(u)du=\overline{\gamma}_{0}H_{1}(t+s),
\end{equation}
and
\begin{align}
\mathbb{E}[\lambda_{-s}^{\dagger,2}(t)]
&=\overline{\gamma}_{0}\int_{-\infty}^{t}\gamma_{2}(t-u)H_{1}(u+s)du
\\
&=\overline{\gamma}_{0}\int_{-\infty}^{t+s}\gamma_{2}(t+s-u)H_{1}(u)du
=\overline{\gamma}_{0}(\gamma_{2}\ast H_{1})(t+s).
\nonumber
\end{align}
Iteratively, we get
\begin{equation}
\mathbb{E}[\lambda_{-s}^{\dagger,n}(t)]
=\overline{\gamma}_{0}
\left(\gamma_{n}\ast\cdots\ast\gamma_{2}\ast H_{1}\right)(t+s).
\end{equation}

Therefore, for any $T>0$,
\begin{align}
\mathbb{P}(N_{-s}^{\dagger}(0,T)>0)
&=1-\mathbb{E}\left[\exp\left(-\int_{0}^{T}\lambda_{s}^{\dagger}(t)dt\right)\right]
\\
&\leq\mathbb{E}\left[\int_{0}^{T}\lambda^{\dagger}_{-s}(t)dt\right]
\nonumber
\\
&=\int_{0}^{T}\sum_{n=1}^{\infty}\mathbb{E}[\lambda_{-s}^{\dagger,n}(t)]dt,
\nonumber
\end{align}
where $\lambda_{s}^{\dagger}(t):=\sum_{n=1}^{\infty}\lambda_{-s}^{\dagger,n}(t)$ is the intensity of $N_{-s}^{\dagger}$ at time $t$.

Since $\gamma_{n}(t)$ is integrable for any $n$, $H_{1}(t)\rightarrow 0$ as $t\rightarrow\infty$.
Thus, $\mathbb{E}[\lambda_{-s}^{\dagger,n}(t)]\rightarrow 0$ as $s\rightarrow\infty$
for any $t$. Moreover,
\begin{equation}
\mathbb{E}[\lambda_{-s}^{\dagger,n}(t)]
\leq\overline{\gamma}_{0}
H_{1}(0)\left(\gamma_{n}\ast\cdots\ast\gamma_{2}\ast 1\right)(t+s),
\end{equation}
where $1$ stands for the function from $\mathbb{R}^{+}$ to $\mathbb{R}$ that takes the constant value $1$.
Thus, for any $t$,
\begin{equation}
\limsup_{s\rightarrow\infty}\mathbb{E}[\lambda_{-s}^{\dagger,n}(t)]
\leq\overline{\gamma}_{0}\prod_{i=1}^{n}\Vert\gamma_{i}\Vert_{L^{1}}
\leq\overline{\gamma}_{0}\rho^{n},
\end{equation}
by Assumption \ref{MainAssumption} (ii), which is summable in $n$. Therefore, for any $T>0$,
\begin{equation}
\mathbb{P}(N_{-s}^{\dagger}(0,T)>0)
\leq\int_{0}^{T}\sum_{n=1}^{\infty}\mathbb{E}[\lambda_{-s}^{\dagger,n}(t)]dt\rightarrow 0,
\end{equation}
as $s\rightarrow\infty$. Hence, we proved the weak asymptotic stationarity of $N$.

(ii) Since $\int_{0}^{\infty}t\gamma_{1}(t)dt<\infty$, 
\begin{align}
\int_{0}^{\infty}(\gamma_{n}\ast\gamma_{n-1}\ast\cdots\gamma_{2}\ast H_{1})(t+s)dt
&\leq\int_{0}^{\infty}(\gamma_{n}\ast\gamma_{n-1}\ast\cdots\gamma_{2}\ast H_{1})(t)dt
\\
&=\Vert\gamma_{n}\Vert_{L^{1}}\cdots\Vert\gamma_{j+2}\Vert_{L^{1}}\Vert H_{1}\Vert_{L^{1}}<\infty,
\nonumber
\end{align}
since $\Vert H_{1}\Vert_{L^{1}}=\int_{0}^{\infty}\int_{t}^{\infty}\gamma_{1}(s)dsdt
=\int_{0}^{\infty}t\gamma_{1}(t)dt<\infty$. 
Together with Assumption \ref{MainAssumption} (ii)
and the proofs in part (i), we get
\begin{equation}
\int_{0}^{\infty}\sum_{n=1}^{\infty}\mathbb{E}[\lambda_{-s}^{\dagger,n}(t)]dt
\leq\overline{\gamma}_{0}\sum_{n=1}^{\infty}\rho^{n-1}\int_{0}^{\infty}t\gamma_{1}(t)dt<\infty,
\end{equation}
and therefore
\begin{equation}
\mathbb{P}(N_{-s}^{\dagger}(0,\infty)>0)
\leq\int_{0}^{\infty}\sum_{n=1}^{\infty}\mathbb{E}[\lambda_{-s}^{\dagger,n}(t)]dt\rightarrow 0,
\end{equation}
as $s\rightarrow\infty$. Hence, we proved the strong asymptotic stationarity of $N$.
\end{proof}

\begin{remark}
In Theorem \ref{ergodicity}, we assumed that $\gamma_{0}(t)\equiv\overline{\gamma}_{0}$ being a constant.
It will be interesting to extend the convergence to equilibrium results in Theorem \ref{ergodicity} under a weaker assumption.
\end{remark}

\section{Connections to Multivariate Hawkes Processes}\label{ConnectionSection}

In this section, we will show that the Hawkes process with different exciting functions
is related to the multivariate Hawkes process, 
see e.g. Hawkes \cite{Hawkes}, Liniger \cite{Liniger}, Bacry et al. \cite{Bacry}.
A multivariate Hawkes process is multidimensional point process $(N_{1}(t),\ldots,N_{d}(t))$
such that for any $1\leq i\leq d$, $N_{i}(t)$ is a simple point process with intensity
\begin{equation}
\lambda_{i}(t):=\nu_{i}+\sum_{j=1}^{d}\int_{0}^{t}\phi_{ij}(t-s)N_{j}(ds),
\end{equation}
where $\nu_{i}$ are non-negative constants and $\phi_{ij}(t)$ are non-negative real-valued functions, 
and $\Vert\phi_{ij}\Vert_{L^{1}}<\infty$. 
If the spectral radius of the matrix $(\Vert\phi_{ij}\Vert_{L^{1}})_{1\leq i,j\leq d}$ is less than $1$,
then, we have the law of large numbers, see e.g. Bacry et al. \cite{Bacry}
\begin{equation}\label{MultiLLN}
\frac{1}{t}(N_{1}(t),\ldots,N_{d}(t))^{t}\rightarrow(I-\Phi)^{-1}\nu,
\end{equation}
as $t\rightarrow\infty$ where $\nu=(\nu_{1},\ldots,\nu_{d})^{t}$ 
and $\Phi=(\Vert\phi_{ij}\Vert_{L^{1}})_{1\leq i,j\leq d}$.

Let us consider a special case of the Hawkes process with 
exciting functions $(\gamma_{n})_{n\in\mathbb{N}\cup\{0\}}$ by letting
$\gamma_{0}(t)\equiv\overline{\gamma}_{0}$, 
$\gamma_{n}(t)=h(t)$ if $n\in\mathbb{N}$ is odd and $\gamma_{n}(t)=g(t)$ if $n\in\mathbb{N}$
is even. We can consider two mutually exciting processes $N^{\text{even}}$ and $N^{\text{odd}}$
defined as
\begin{equation}
N^{\text{even}}:=\sum_{n=0}^{\infty}N^{2n}
\qquad
\text{and}
\qquad
N^{\text{odd}}:=\sum_{n=0}^{\infty}N^{2n+1}.
\end{equation}
$N^{\text{even}}$ and $N^{\text{odd}}$ are mutually exciting since $N^{n}$ is generated
based on $N^{n-1}$ and a jump in $N^{2n}$ will lead to more jumps for $N^{2n+1}$
and a jump in $N^{2n+1}$ will on the other hand contribute to more jumps for $N^{2n+2}$.
By the law of large numbers result due to Fierro et al. \cite{Fierro},
\begin{equation}
\frac{N_{t}}{t}=\frac{N^{\text{even}}_{t}}{t}+\frac{N^{\text{odd}}_{t}}{t}
\rightarrow m
\end{equation}
a.s. as $t\rightarrow\infty$, where
\begin{align}\label{evenoddlimit}
m&=\overline{\gamma}_{0}\sum_{n=1}^{\infty}\prod_{i=1}^{n}\Vert\gamma_{i}\Vert_{L^{1}}
\\
&=\overline{\gamma}_{0}\left(\Vert h\Vert_{L^{1}}+\Vert h\Vert_{L^{1}}\Vert g\Vert_{L^{1}}
+\Vert h\Vert_{L^{1}}\Vert g\Vert_{L^{1}}\Vert h\Vert_{L^{1}}+\cdots\right)
\nonumber
\\
&=\frac{1+\Vert h\Vert_{L^{1}}}{1-\Vert h\Vert_{L^{1}}\Vert g\Vert_{L^{1}}}.
\nonumber
\end{align}
Now, let us point out the connections to the multivariate Hawkes process.
The intensity of $N^{\text{even}}$ is given by
\begin{align}
\lambda^{\text{even}}_{t}
&=\sum_{n=0}^{\infty}\lambda^{2n}_{t}
\\
&=\overline{\gamma}_{0}+\sum_{n=1}^{\infty}\int_{0}^{t}h(t-s)N^{2n-1}(ds)
\nonumber
\\
&=\overline{\gamma}_{0}+\int_{0}^{t}h(t-s)N^{\text{odd}}(ds).
\nonumber
\end{align}
Similarly, the intensity of $N^{\text{odd}}$ is given by
\begin{equation}
\lambda^{\text{odd}}_{t}=\int_{0}^{t}g(t-s)N^{\text{even}}(ds).
\end{equation}
Therefore, $(N^{\text{even}}_{t},N^{\text{odd}}_{t})$ is a bivariate Hawkes process with
\begin{equation}
\nu=
\left(
\begin{array}{c}
\overline{\gamma}_{0}
\\
0
\end{array}
\right)
\qquad
\text{and}
\qquad
\Phi=
\left(
\begin{array}{cc}
0 & \Vert h\Vert_{L^{1}}
\\
\Vert g\Vert_{L^{1}} & 0
\end{array}
\right).
\end{equation}
Thus, by \eqref{MultiLLN},
\begin{equation}
\frac{1}{t}\left(
\begin{array}{c}
N^{\text{even}}_{t}
\\
N^{\text{odd}}_{t}
\end{array}
\right)
\rightarrow
\left(
\begin{array}{cc}
1 & -\Vert h\Vert_{L^{1}}
\\
-\Vert g\Vert_{L^{1}} & 1
\end{array}
\right)^{-1}
\left(
\begin{array}{c}
\overline{\gamma}_{0}
\\
0
\end{array}
\right)
=\left(
\begin{array}{c}
\frac{\overline{\gamma}_{0}}{1-\Vert h\Vert_{L^{1}}\Vert g\Vert_{L^{1}}}
\\
\frac{\overline{\gamma}_{0}\Vert h\Vert_{L^{1}}}{1-\Vert h\Vert_{L^{1}}\Vert g\Vert_{L^{1}}}
\end{array}
\right),
\end{equation}
as $t\rightarrow\infty$, which is consistent with \eqref{evenoddlimit}.

Indeed, we can work in a more generating setting. Let $(A_{i})_{i=1}^{d}$ be a partition
of $\mathbb{N}\cup\{0\}$, i.e. $A_{i}\cap A_{j}=\emptyset$ for any $i\neq j$
and $\cup_{i=1}^{d}A_{i}=\mathbb{N}\cup\{0\}$. 
Assume that $\gamma_{0}\equiv\overline{\gamma}_{0}$ and $(\gamma_{n})_{n\in\mathbb{N}}$
may not be homogeneous. Define
\begin{equation}
N_{1}=\sum_{n\in A_{1}}N^{n},
\qquad
N_{2}=\sum_{n\in A_{2}}N^{n},
\qquad
\cdots\cdots
\qquad
N_{d}=\sum_{n\in A_{d}}N^{n}.
\end{equation}
Then, the $d$-dimensional process $(N_{1},\ldots,N_{d})$ has the mutually exciting property
and it is more general than the classical multivariate Hawkes process.
Since we proved convergence to equilibrium in Theorem \ref{ergodicity}, by ergodic theorem,
\begin{equation}
\frac{1}{t}(N_{1},N_{2},\ldots,N_{d})\rightarrow
\left(\sum_{n\in A_{1}}m_{n},\sum_{n\in A_{2}}m_{n},\ldots,\sum_{n\in A_{d}}m_{n}\right),
\end{equation}
a.s. as $t\rightarrow\infty$, where $m_{n}$ is defined in \eqref{mn}.

\section{Moderate and Large Deviations}\label{MDPLDPSection}

In this section, we are interested to study the moderate and large deviations for $\mathbb{P}(\frac{N_{t}}{t}\in\cdot)$.
The large deviations for classical Hawkes processes have been well studied in the literature 
for both linear and nonlinear cases, see e.g. Bordenave and Torrisi \cite{Bordenave},
Zhu \cite{ZhuI} and Zhu \cite{ZhuII}. The moderate deviations for classical Hawkes processes
have been studied for the linear case, see e.g. Zhu \cite{ZhuMDP}.

In the linear case, let us assume that
\begin{equation}
\lambda_{t}=\nu+\int_{0}^{t}h(t-s)N(ds),
\end{equation}
where $\Vert h\Vert_{L^{1}}<1$ and $\int_{0}^{\infty}th(t)dt<\infty$.
Bordenave and Torrisi \cite{Bordenave} proved a large deviation principle for $\mathbb{P}(\frac{N_{t}}{t}\in\cdot)$ with the rate function
\begin{equation}\label{classicalLDP}
I(x)=
\begin{cases}
x\log\left(\frac{x}{\nu+x\Vert h\Vert_{L^{1}}}\right)-x+x\Vert h\Vert_{L^{1}}+\nu &\text{if $x\in[0,\infty)$}
\\
+\infty &\text{otherwise}
\end{cases}.
\end{equation}
Moreover, Karabash and Zhu \cite{Karabash} obtained a large deviation principle
for the linear Hawkes process with random marks.

For nonlinear Hawkes processes, i.e. when $\lambda(\cdot)$ is nonlinear, Zhu \cite{ZhuI} 
first considered the case
that $h(\cdot)$ is exponential, i.e. when the Hawkes process is Markovian 
and obtained a large deviation principle for $\mathbb{P}(N_{t}/t\in\cdot)$
Then, Zhu \cite{ZhuI} also proved the large deviation principle for the case when $h(\cdot)$ 
is a sum of exponentials and used that as an approximation to recover the result 
for the linear case proved in Bordenave and Torrisi \cite{Bordenave} and also
for a special class of general nonlinear Hawkes processes.
For the most general $h(\cdot)$ and $\lambda(\cdot)$, Zhu \cite{ZhuII} proved a process-level, 
i.e. level-3 large deviation principle for the Hawkes process and used contraction 
principle to obtain a large deviation principle for $\mathbb{P}(N_{t}/t\in\cdot)$.

The large deviations result for $(N_{t}/t\in\cdot)$ is helpful to study the ruin probabilities 
of a risk process when the claims arrivals follow a Hawkes process. 
Stabile and Torrisi \cite{Stabile} considered risk processes with non-stationary 
Hawkes claims arrivals and studied the asymptotic 
behavior of infinite and finite horizon ruin probabilities under light-tailed 
conditions on the claims.
The corresponding result for heavy-tailed claims was obtained by Zhu \cite{ZhuRuin}. 

Before we proceed, let us recall that
a sequence of probability measures $(P_{n})_{n\in\mathbb{N}}$ 
on a topological space $X$ satisfies
a large deviation principle with speed $n$ and rate function $I:X\rightarrow\mathbb{R}$ 
if $I$ is non-negative,
lower semicontinuous and for any measurable set $A$, 
\begin{equation}
-\inf_{x\in A^{o}}I(x)\leq\liminf_{n\rightarrow\infty}\frac{1}{n}\log P_{n}(A)
\leq\limsup_{n\rightarrow\infty}\frac{1}{n}\log P_{n}(A)\leq-\inf_{x\in\overline{A}}I(x).
\end{equation}
Here, $A^{o}$ is the interior of $A$ and $\overline{A}$ is its closure. 
We refer to Dembo and Zeitouni \cite{Dembo} or Varadhan \cite{VaradhanII} 
for general background of large deviations and the applications.

\begin{theorem}\label{LogLimit}
 Let $\Gammaset$ denote $(\Vert\gamma_{i}\Vert_{L^{1}})_{i\in\mathbb{N}}$. Under Assumption \ref{MainAssumption} and $N(-\infty,0]=0$,
for any $\theta\in\mathbb{R}$, $\Gamma(\theta):=\lim_{t\rightarrow\infty}
\frac{1}{t}\log\mathbb{E}[e^{\theta N_{t}}]$
exists and
\begin{equation}
\Gamma(\theta)=\overline{\gamma}_{0}(e^{f_{\infty}(\Gammaset,\theta)}-1),
\end{equation}
where $f_{\infty}(\Gammaset,\theta):=\lim_{M\rightarrow\infty}f(M,M,\Gammaset,\theta)$ exists 
on the extended real line and 
$f(n,M,\Gammaset,\theta)$ is define recursively as
\begin{equation}\label{eq:f_infinity}
f(n,M,\Gammaset,\theta)=\theta+\Vert\gamma_{M+1-n}\Vert_{L^{1}}(e^{f(n-1,M,\Gammaset,\theta)}-1),
\qquad 1\leq n\leq M,
\end{equation}
with $f(0,M,\Gammaset,\theta)=0$.
\end{theorem}

\begin{remark}
It is easy to compute that
\begin{equation}
\frac{\partial}{\partial\theta}f(n,M,\Gammaset,\theta)
=1+\Vert\gamma_{M+1-n}\Vert_{L^{1}}e^{f(n-1,M,\Gammaset,\theta)}\frac{\partial}{\partial\theta}f(n-1,M,\Gammaset,\theta).
\end{equation}
Now, note that $f(n,M,\Gammaset,\theta)|_{\theta=0}=0$ for every $n$, $M$, and $\Gammaset$.
By iterating and setting $\theta=0$, we get
\begin{equation}
\frac{\partial}{\partial\theta}f(M,M,\Gammaset,\theta)\bigg|_{\theta=0}
=1+\sum_{p=1}^{M}\prod_{j=1}^{p}\Vert\gamma_{j}\Vert_{L^{1}}=\sum_{p=0}^{M}m_{p},
\end{equation}
and as $M$ goes to $\infty$, we get 
$\lim_{M\rightarrow\infty}\frac{\partial}{\partial\theta}f(M,M,\Gammaset,\theta)|_{\theta=0}
=\sum_{p=0}^{M}m_{p}$, which is consistent with the law of large numbers \eqref{LLN}.
\end{remark}

\begin{remark}
In the case of classical linear Hawkes process, say $\lambda_{t}=\nu+\int_{0}^{t}h(t-s)N(ds)$,
it is easy to see that $\overline{\gamma}_{0}=\nu$ and $\gamma_{n}=h$ for any $n\in\mathbb{N}$.
Thus $\Gamma(\theta)=\nu(f(\theta)-1)$, if $\theta\leq\Vert h\Vert_{L^{1}}-\log\Vert h\Vert_{L^{1}}-1$
and $\Gamma(\theta)=\infty$ otherwise,
where $f(\theta)$ is the smaller solution
of the two solutions of the equation $f(\theta)=e^{\theta+\Vert h\Vert_{L^{1}}(f(\theta)-1)}$. Then, it is easy
to check that $I(x)=\sup_{\theta\in\mathbb{R}}\{\theta x-\Gamma(\theta)\}$ gives \eqref{classicalLDP}.
More generally, for example, 
if we assume that $\gamma_{n}=h$ for odd $n\in\mathbb{N}$ and $\gamma_{n}=g$ for even
$n\in\mathbb{N}$, then, $\Gamma(\theta)=\overline{\gamma}_{0}(f(\theta)-1)$
for $\theta\leq\theta_{c}$ and $\Gamma(\theta)=\infty$ otherwise
\footnote{Let $F(x,\theta)=x-e^{\theta+\Vert h\Vert_{L^{1}}(e^{\theta+\Vert g\Vert_{L^{1}}(x-1)}-1)}$.
Note that $F(x,\theta)$ has two roots when $\theta$ is less than a critical value.
The critical value $\theta_{c}$ and $x_{c}$ are determined via
$F(x_{c},\theta_{c})=0$ and $\frac{\partial}{\partial x}F(x_{c},\theta_{c})=0$, 
which implies that $x_{c}=e^{\theta_{c}+\Vert h\Vert_{L^{1}}(e^{\theta_{c}+\Vert g\Vert_{L^{1}}(x_{c}-1)}-1)}$
and $1=\Vert h\Vert_{L^{1}}\Vert g\Vert_{L^{1}}e^{\theta_{c}+\Vert g\Vert_{L^{1}}(x_{c}-1)}x_{c}$.
The second identity gives an expression of $\theta_{c}$ in terms of $x_{c}$ and
substitute into the first identity it gives an equation that determines $x_{c}$.
To see $\Vert h\Vert_{L^{1}}\Vert g\Vert_{L^{1}}x_{c}^{2}=e^{-\Vert g\Vert_{L^{1}}(x_{c}-1)+\frac{1}{\Vert g\Vert_{L^{1}}x_{c}}-\Vert h\Vert_{L^{1}}}$ has a unique solution greater than $1$, we notice that LHS of this equation is increasing in $x_{c}$ 
and RHS is decreasing in $x_{c}$, and LHS increases to $\infty$ as $x_{c}\uparrow\infty$
and RHS decreases to $0$ as $x_{c}\uparrow\infty$. Moreover at $1$,  
$\Vert h\Vert_{L^{1}}\Vert g\Vert_{L^{1}}<e^{\frac{1}{\Vert g\Vert_{L^{1}}}-\Vert h\Vert_{L^{1}}}$
since $\Vert h\Vert_{L^{1}}e^{\Vert h\Vert_{L^{1}}}<\frac{1}{\Vert g\Vert_{L^{1}}}e^{\frac{1}{\Vert g\Vert_{L^{1}}}}$
since $\Vert h\Vert_{L^{1}}<1<\frac{1}{\Vert g\Vert_{L^{1}}}$.}, 
where $\theta_{c}=-\Vert g\Vert_{L^{1}}(x_{c}-1)+\log(\frac{1}{\Vert h\Vert_{L^{1}}\Vert g\Vert_{L^{1}}x_{c}})$,
and $x_{c}$ is the unique value greater than $1$ that satisfies
$\Vert h\Vert_{L^{1}}\Vert g\Vert_{L^{1}}x_{c}^{2}=e^{-\Vert g\Vert_{L^{1}}(x_{c}-1)+\frac{1}{\Vert g\Vert_{L^{1}}x_{c}}-\Vert h\Vert_{L^{1}}}$, and
$f(\theta)$ is the smaller solution that satisfies
\begin{equation}
f(\theta)=e^{\theta+\Vert h\Vert_{L^{1}}(e^{\theta+\Vert g\Vert_{L^{1}}(f(\theta)-1)}-1)}.
\end{equation}
\end{remark}

\begin{proof}[Proof of Theorem \ref{LogLimit}]
For any $M\in\mathbb{N}$, $\theta\in\mathbb{R}$, and continuous deterministic function $G(s)$, $0\leq s\leq t$,
\begin{align}
&\mathbb{E}\left[e^{\int_{0}^{t}G(t-s)N^{M}(ds)+\theta\sum_{n=0}^{M-1}N^{n}_{t}}\right]
\\
&=\mathbb{E}\left[\mathbb{E}\left[e^{\int_{0}^{t}G(t-s)N^{M}(ds)}\big|N^{0},N^{1},\ldots,N^{M-1}\right]e^{\theta\sum_{n=0}^{M-1}N^{n}_{t}}\right]
\nonumber
\\
&=\mathbb{E}\left[e^{\int_{0}^{t}(e^{G(t-s)}-1)\lambda^{M}_{s}ds}e^{\theta\sum_{n=0}^{M-1}N^{n}_{t}}\right]
\nonumber
\\
&=\mathbb{E}\left[e^{\int_{0}^{t}(e^{G(t-s)}-1)\int_{0}^{s}\gamma_{M}(s-u)N^{M-1}(du)ds}e^{\theta\sum_{n=0}^{M-1}N^{n}_{t}}\right]
\nonumber
\\
&=\mathbb{E}\left[e^{\int_{0}^{t}[\int_{u}^{t}(e^{G(t-s)}-1)\gamma_{M}(s-u)ds]N^{M-1}(du)}e^{\theta\sum_{n=0}^{M-1}N^{n}_{t}}\right]
\nonumber
\\
&=\mathbb{E}\left[e^{\int_{0}^{t}[\int_{0}^{t-u}(e^{G(t-u-s)}-1)\gamma_{M}(s)ds]N^{M-1}(du)}e^{\theta\sum_{n=0}^{M-1}N^{n}_{t}}\right]
\nonumber
\end{align}
Therefore, we have for any $M\in\mathbb{N}$ and $\theta\in\mathbb{R}$,
\begin{equation}
\mathbb{E}\left[e^{\theta\sum_{n=0}^{M}N^{n}_{t}}\right]
=e^{\int_{0}^{t}(e^{f(M,M,\Gammaset,\theta,t-s)}-1)\gamma_{0}(s)ds},
\end{equation}
where $f(\cdot,\cdot,\cdot,\cdot,\cdot)$ is defined recursively as
\begin{equation}\label{Gn}
f(n,M,\Gammaset,\theta,t)=\theta+\int_{0}^{t}(e^{f(n-1,M,\Gammaset,\theta,t-s)}-1)\gamma_{M+1-n}(s)ds,
\qquad 1\leq n\leq M-1,
\end{equation}
and $f(0,M,\Gammaset,\theta,t)=0$, where $\Gammaset$ was $(\Vert\gamma_{i}\Vert_{L^{1}})_{i\in\mathbb{N}}$.

It is easy to see that for any given $M\in\mathbb{N}$,
\begin{equation}
\lim_{t\rightarrow\infty}f(n,M,\Gammaset,\theta,t)
=:f(n,M,\Gammaset,\theta),
\end{equation}
where $f(0,M,\Gammaset,\theta)=0$ and
\begin{equation}\label{compareHawkes}
f(n,M,\Gammaset,\theta)=\theta+\Vert\gamma_{M+1-n}\Vert_{L^{1}}(e^{f(n-1,M,\Gammaset,\theta)}-1).
\end{equation}

Since for $\theta\geq 0$, $e^{\theta\sum_{n=0}^{M}N^{n}_{t}}$ is increasing in $M$
and for $\theta<0$, it is decreasing in $M$, by monotone convergence theorem,
\begin{equation}
\mathbb{E}[e^{\theta N_{t}}]=\lim_{M\rightarrow\infty}\mathbb{E}\left[e^{\theta\sum_{n=0}^{M}N^{n}_{t}}\right]
=e^{\int_{0}^{t}(e^{\lim_{M\rightarrow\infty}f(M,M,\Gammaset,\theta,t-s)}-1)\gamma_{0}(s)ds}.
\end{equation}
Since for $\theta\geq 0$, $f(M,M,\Gammaset,\theta,t)$ is increasing in both $M$ and $t$ 
and for $\theta<0$, $f(M,M,\Gammaset,\theta,t)$ is decreasing
in both $M$ and $t$, we have
\begin{equation}
\lim_{t\rightarrow\infty}\lim_{M\rightarrow\infty}f(M,M,\Gammaset,\theta,t)
=\lim_{M\rightarrow\infty}\lim_{t\rightarrow\infty}f(M,M,\Gammaset,\theta,t)=\lim_{M\rightarrow\infty}f(M,M,\Gammaset,\theta),
\end{equation}
and for any $\theta<0$, $f(M,M,\Gammaset,\theta)$ is decreasing in $M$ and 
$f_{\infty}(\Gammaset,\theta):=\lim_{M\rightarrow\infty}f(M,M,\Gammaset,\theta)$
exists. For any $\theta\geq 0$, $f(M,M,\Gammaset,\theta)$ is increasing in $M$ 
and the limit $f_{\infty}(\Gammaset,\theta):=\lim_{M\rightarrow\infty}f(M,M,\Gammaset,\theta)$ 
exists on extended positive real line $[0,\infty]$.
Hence, we conclude that
\begin{equation}
\lim_{t\rightarrow\infty}\frac{1}{t}\log\mathbb{E}[e^{\theta N_{t}}]=\overline{\gamma}_{0}(e^{f_{\infty}(\Gammaset,\theta)}-1)
\end{equation}
exists on the extended real line.
\end{proof}

\begin{theorem}\label{LDPThm}
Under Assumption \ref{MainAssumption} and $N(-\infty,0]=0$,
$\mathbb{P}(N_{t}/t\in\cdot)$ satisfies a large deviation principle with rate function
\begin{equation}
I(x):=\sup_{\theta\in\mathbb{R}}\{\theta x-\Gamma(\theta)\}.
\end{equation}
\end{theorem}

\begin{proof}
Because we already had Theorem \ref{LogLimit}, we can apply G\"{a}rtner-Ellis theorem
to obtain the large deviation principle if we can check the essential smoothness condition.

Let us defined the set
\begin{equation}
\mathcal{D}_{\Gamma}:=\{\theta:\Gamma(\theta)<\infty\}.
\end{equation}
Note that \eqref{eq:f_infinity} only depends on the $L^1$ norms of $\gamma_n$, and recall that we assumed
$\rho:=\sup_{n\in\mathbb{N}}\int_{0}^{\infty}\gamma_{n}(t)dt<1$.
For a classical linear Hawkes process with immigration rate $\nu$ and exciting function $h(t)$ and $\rho = \Vert h\Vert_{L^{1}}<1$.
The limit $\lim_{t\rightarrow\infty}\frac{1}{t}\log\mathbb{E}[e^{\theta N_{t}}]$ exists and is finite
for any $\theta\leq\Vert h\Vert_{L^{1}}-1-\log\Vert h\Vert_{L^{1}}$.
By comparing with the classical Hawkes process
(by using $\Vert\gamma_{n}\Vert_{L^{1}}\leq\rho$ for every $n$
and \eqref{compareHawkes}), there exists some constant $\theta_{c}\geq\rho-1-\log\rho>0$ so that
for any $\theta\leq\theta_{c}$, $\Gamma(\theta)<\infty$. More precisely, let us define
$\theta_{c}:=\sup\{\theta:\Gamma(\theta)<\infty\}$.
Hence, we showed that the interior of $\mathcal{D}_{\Gamma}$ contains a nonempty neighborhood of the origin.

Next, we need to show that for any $\theta<\theta_{c}$, $\Gamma(\theta)$ is differentiable at $\theta$.
Let $S_{t}:=\sum_{n=1}^{\infty}N_{t}^{n}$. A quick look at the proof of Theorem \ref{LogLimit}
reveals that
\begin{equation}
\psi(\theta)=\lim_{t\rightarrow\infty}\mathbb{E}\left[e^{\theta S_{t}}\right]
=e^{f_{\infty}(\Gamma,\theta)}.
\end{equation}
Now note that $S_{t}$ is always positive and so is $e^{\theta S_{t}}$.
By dominated convergence theorem, $\psi(\theta)$ is twice differentiable
inside $(-\infty,\theta_{c})$ and the derivatives are
\begin{equation}
\lim_{t\rightarrow\infty}\mathbb{E}\left[S_{t}e^{\theta S_{t}}\right]
\qquad\text{and}\qquad
\lim_{t\rightarrow\infty}\mathbb{E}\left[S_{t}^{2}e^{\theta S_{t}}\right].
\end{equation}
Thus we proved the differentiability of $f_{\infty}(\Gammaset,\theta)$ and $\Gamma(\theta)$ inside
the domain.

Finally, let us prove steepness. 
Let us recall that 
\begin{equation}
f(n,M,\Gammaset,\theta)=\theta+\Vert\gamma_{M+1-n}\Vert_{L^{1}}(e^{f(n-1,M,\Gammaset,\theta)}-1),
\end{equation}
and $\Gamma(\theta)=\overline{\gamma}_{0}(e^{f_{\infty}(\Gammaset,\theta)}-1)$, 
where $f_{\infty}(\Gammaset,\theta)=\lim_{M\rightarrow\infty}f(M,M,\Gammaset,\theta)$.
For any $0<\theta<\theta_{c}$,
$f(n,M,\Gammaset,\theta)$ is increasing in $\theta$ for any $n,M\in\mathbb{N}$. Thus
\begin{equation}\label{eq:difflimit}
\frac{\partial}{\partial\theta}f(n,M,\Gammaset,\theta)
=1+\Vert\gamma_{1}\Vert_{L^{1}}e^{f(n-1,M,\Gammaset,\theta)}\frac{\partial}{\partial\theta}f(n-1,M,\Gammaset,\theta)\geq 1.
\end{equation}
Using \eqref{eq:difflimit} for $n=M-1,M$, we have
\begin{equation}
\frac{\partial}{\partial\theta}f(M,M,\Gammaset,\theta)
\geq 1+\Vert\gamma_{1}\Vert_{L^{1}}e^{f(M-1,M,\Gammaset,\theta)}\rightarrow+\infty,
\end{equation}
as $\theta\uparrow\theta_{c}$. Thus we proved steepness. The proof is complete.
\end{proof}

Let $C_{1},C_{2},\ldots$ be a sequence of real-valued i.i.d. random variables
with finite mean $\mathbb{E}[C_{1}]$ and variance $\text{Var}[C_{1}]$, independent of
the point process $N_{t}$. Fierro et al. \cite{Fierro} showed that
\begin{equation}
\frac{\sum_{i=1}^{N_{t}}C_{i}-\mathbb{E}[C_{1}]\mathbb{E}[N_{t}]}{\sqrt{t}}
\rightarrow 
N\left(0,m\text{Var}[C_{1}]+\mathbb{E}[C_{1}]\sigma^{2}\right),
\end{equation}
in distribution as $t\rightarrow\infty$.

We have the following result.
\begin{theorem}\label{CLDPThm}
Assume $N(-\infty,0]=0$ and Assumption \ref{MainAssumption}.
Further assume that $\mathbb{E}[e^{\theta C_{1}}]<\infty$ for $\theta\in(-\gamma,\gamma)$ for some $\gamma>0$.
Then, the limit $\Gamma_{C}(\theta):=\lim_{t\rightarrow\infty}
\frac{1}{t}\log\mathbb{E}[e^{\theta\sum_{i=1}^{N_{t}}C_{i}}]$ exists and
indeed $\Gamma_{C}(\theta)=\Gamma(\log\mathbb{E}[e^{\theta C_{1}}])$, where $\Gamma(\cdot)$
is defined in Theorem \ref{LogLimit}. Moreover, $\mathbb{P}(\frac{1}{t}\sum_{i=1}^{N_{t}}C_{i}\in\cdot)$
satisfies a large deviation principle with rate function
\begin{equation}\label{CRateFunction}
I_{C}=\sup_{\theta\in\mathbb{R}}\{\theta x-\Gamma_{C}(\theta)\}.
\end{equation}
\end{theorem}

\begin{proof}
For any $\theta\in\mathbb{R}$ so that $\mathbb{E}\left[e^{\theta\sum_{i=1}^{N_{t}}C_{i}}\right]<\infty$, we have
\begin{align}
\mathbb{E}\left[e^{\theta\sum_{i=1}^{N_{t}}C_{i}}\right]
&=\mathbb{E}\left[\mathbb{E}\left[e^{\theta\sum_{i=1}^{N_{t}}C_{i}}|N_{t}\right]\right]
\\
&=\sum_{k=0}^{\infty}\mathbb{E}[e^{\theta\sum_{i=1}^{k}C_{i}}]\mathbb{P}(N_{t}=k)
\nonumber
\\
&=\sum_{k=0}^{\infty}e^{k\log\mathbb{E}[e^{\theta C_{1}}]}\mathbb{P}(N_{t}=k)
\nonumber
\\
&=\mathbb{E}\left[e^{\log\mathbb{E}[e^{\theta C_{1}}]N_{t}}\right].
\nonumber
\end{align}
Hence, by Theorem \ref{LogLimit}, 
we get $\Gamma_{C}(\theta):=\lim_{t\rightarrow\infty}\frac{1}{t}\log\mathbb{E}[e^{\theta\sum_{i=1}^{N_{t}}C_{i}}]
=\Gamma(\log\mathbb{E}[e^{\theta C_{1}}])$. Following the proof of Theorem \ref{LDPThm},
we conclude that $\mathbb{P}(\frac{1}{t}\sum_{i=1}^{N_{t}}C_{i}\in\cdot)$
satisfies a large deviation principle with rate function $I_{C}(x)$ given by \eqref{CRateFunction}.
\end{proof}

Let $X_{1},\ldots,X_{n}$ be a sequence of real-valued i.i.d. random variables with mean $0$ and variance $\sigma^{2}$. 
Assume that $\mathbb{E}[e^{\theta X_{1}}]<\infty$ for $\theta$ in some ball around
the origin. For any $\sqrt{n}\ll a_{n}\ll n$, a moderate deviation principle says that for any Borel set $A$,
\begin{align}
-\inf_{x\in A^{o}}\frac{x^{2}}{2\sigma^{2}}
&\leq\liminf_{n\rightarrow\infty}\frac{n}{a_{n}^{2}}\log\mathbb{P}\left(\frac{1}{a_{n}}\sum_{i=1}^{n}X_{i}\in A\right)
\\
&\leq\limsup_{n\rightarrow\infty}\frac{n}{a_{n}^{2}}\log\mathbb{P}\left(\frac{1}{a_{n}}\sum_{i=1}^{n}X_{i}\in A\right)
\leq-\inf_{x\in A^{o}}\frac{x^{2}}{2\sigma^{2}}.\nonumber
\end{align}
In other words, $\mathbb{P}(\frac{1}{a_{n}}\sum_{i=1}^{n}X_{i}\in\cdot)$ satisfies a large deviation principle with the speed $\frac{a_{n}^{2}}{n}$.
The above classical result can be found for example in Dembo and Zeitouni \cite{Dembo}.
Moderate deviation principle fills in the gap between central limit theorem and large deviation principle.

The moderate deviation principle for classical linear Hawkes process has been studied in Zhu \cite{ZhuMDP}.
For the remaining of this section, let us prove the moderate deviation principle for the Hawkes process
with different exciting functions.

\begin{theorem}\label{MDPThm}
Assume that Assumptions \ref{MainAssumption} and \ref{AssumptionTwo} hold.
For any Borel set $A$ and time sequence $a(t)$ such that $\sqrt{t}\ll a(t)\ll t$, we have
the following moderate deviation principle.
\begin{align}
-\inf_{x\in A^{o}}J(x)&\leq\liminf_{t\rightarrow\infty}\frac{t}{a(t)^{2}}
\log\mathbb{P}\left(\frac{N_{t}-mt}{a(t)}\in A\right)
\\
&\leq\limsup_{t\rightarrow\infty}\frac{t}{a(t)^{2}}
\log\mathbb{P}\left(\frac{N_{t}-mt}{a(t)}\in A\right)
\leq-\inf_{x\in\overline{A}}J(x),\nonumber
\end{align}
where $J(x)=\frac{x^{2}}{2\sigma^{2}}$ and $\sigma^{2}=\sum_{j=0}^{\infty}\left(1+\sum_{p=1}^{\infty}\prod_{i=j+1}^{p+j}\int_{0}^{\infty}\gamma_{i}(u)du\right)^{2}m_{j}$.
\end{theorem}

Before we proceed to the proof of Theorem \ref{MDPThm}, let us prove a series of lemmas.

\begin{lemma}\label{MDPLemma0}
Consider the equation
\begin{equation}\label{minsolution}
x=\theta+(e^{x}-1)\rho.
\end{equation}
The equation has two distinct solutions if $\theta<\rho-1-\log\rho$
and has one solution if $\theta=\rho-1-\log\rho$.
Let $x(\theta)$ be the minimal solution of \eqref{minsolution} if the solutions exist.
Then, $x(\theta)\leq 0$ if $\theta\leq 0$ and $x(\theta)\geq 0$ if $\theta\geq 0$.
\end{lemma}

\begin{proof}
Let $F(x):=x-\theta-(e^{x}-1)\rho$. It is easy to compute that $F'(x)=1-\rho e^{x}$
and $F''(x)=-\rho e^{x}$. Thus, $F(x)$ is strictly concave and its maximum is
achieved at $x=\log(1/\rho)$. Therefore, the equation \eqref{minsolution}
has no solutions if $F(\log(1/\rho))<0$, has one solution if $F(\log(1/\rho))=0$
and has two solutions if $F(\log(1/\rho))>0$. It is easy to check
that $F(\log(1/\rho))=-\theta-\log\rho+\rho-1$. Now assume that $\theta\leq\rho-1-\log\rho$
so that \eqref{minsolution} has solutions. Observe that if $\theta\geq 0$, then $F(0)=-\theta<0$
and $F(\log(1/\rho))\geq 0$, where $\log(1/\rho)>0$, thus $x(\theta)\geq 0$.
Similarly $x(\theta)\leq 0$ when $\theta\leq 0$.
\end{proof}

\begin{lemma}\label{MDPLemma1}
Given $0\leq\theta\leq\rho-1-\log\rho$ and $x(\theta)$ as in lemma \ref{MDPLemma0}, if $f(t,\theta)\leq x(\theta)$ for any $t\geq 0$, 
then $H_{n}(t,\theta)\leq x(\theta)$
uniformly in $t\geq 0$ and $n\in\mathbb{N}$, where
\begin{equation}
H_{n}(t,\theta):=e^{\theta+\int_{0}^{t}(e^{f(t-s,\theta)}-1)\gamma_{n}(s)ds},
\qquad t\geq 0.
\end{equation}
Similarly, given $\theta\leq 0$, if $f(t,\theta)\geq x(\theta)$ for any $t\geq 0$,
then $H_{n}(t,\theta)\geq x(\theta)$.
\end{lemma}

\begin{proof}
Let us first assume that $0\leq\theta\leq\rho-1-\log\rho$.
By the definition of $x(\theta)$ and the assumption $f(t,\theta)\leq x(\theta)$ for any $t\geq 0$, 
it is easy to see that
\begin{align}
H_{n}(t,\theta)&\leq e^{\theta+\int_{0}^{t}(e^{x(\theta)}-1)\gamma_{n}(s)ds}
\\
&\leq e^{\theta+(e^{x(\theta)}-1)\int_{0}^{t}\gamma_{n}(s)ds}
\nonumber
\\
&\leq e^{\theta+(e^{x(\theta)}-1)\Vert\gamma_{n}\Vert_{L^{1}}}
\nonumber
\\
&\leq e^{\theta+(e^{x(\theta)}-1)\rho}
\nonumber
\\
&=x(\theta),
\nonumber
\end{align}
where we used the fact that $x(\theta)\geq 0$ for $\theta\geq 0$ in Lemma \ref{MDPLemma0}.
Similarly, one can show that, given $\theta\leq 0$, if $f(t,\theta)\geq x(\theta)$ for any $t\geq 0$,
then $H_{n}(t,\theta)\geq x(\theta)$.
\end{proof}

\begin{lemma}\label{MDPLemma2}
For any fixed $\theta$ and $a(t)$ as in Theorem \ref{MDPThm}, there is some $k_{1}\geq\frac{1}{1-\rho}$ so that
for any sufficiently large $t$,
\begin{equation}
\left|f\left(n,M,\Gammaset,\frac{a(t)}{t}\theta,s\right)\right|\leq k_{1}\frac{a(t)}{t}|\theta|,
\end{equation}
uniformly for $1\leq n\leq M$, $M\in\mathbb{N}$ and $s\geq 0$, where $f(\cdot,\cdot,\cdot,\cdot,\cdot)$ was defined in \eqref{Gn}.
\end{lemma}

\begin{proof}
Given $\theta\geq 0$, by Lemma \ref{MDPLemma1} and \eqref{Gn}, we have
$f(n,M,\Gammaset,\theta,t)\leq x(\theta)$.
Notice that in Lemma \ref{MDPLemma0}, $x(0)=0$ and $x'(0)=\frac{1}{1-\rho}>1$ since
$\rho<1$. Therefore, for $0\leq\theta\ll 1$, 
there exists some $k_{1}\geq\frac{1}{1-\rho}$ so that $0\leq x(\theta)\leq k_{1}\theta$.
Therefore, $x\left(\frac{a(t)}{t}\theta\right)\leq k_{1}\frac{a(t)}{t}\theta$
for any sufficiently large $t$. Hence, for $\theta\geq 0$, for sufficiently large $t$,
$f\left(n,M,\Gammaset,\frac{a(t)}{t}\theta,s\right)\leq k_{1}\frac{a(t)}{t}|\theta|$ 
uniformly for $1\leq n\leq M$, $M\in\mathbb{N}$ and $s\geq 0$.
Similarly, given $\theta\leq 0$, by Lemma \ref{MDPLemma1}
and the discussions above, $f(n,M,\Gammaset,\theta,t)\geq x(\theta)\geq k_{1}\theta$ for $\theta\leq 0$
and $|\theta|\ll 1$. Hence, we proved the desired result.
\end{proof}

\begin{lemma}\label{MDPLemma3}
Let us define
\begin{equation}\label{C1Defn}
C_{1}(n,M,s):=1+\int_{0}^{s}C_{1}(n-1,M,s-r)\gamma_{M+1-n}(r)dr,
\end{equation}
and
\begin{equation}\label{C2Defn}
C_{2}(n,M,s):=\int_{0}^{s}\left(C_{2}(n-1,M,s-r)+\frac{1}{2}[C_{1}(n-1,M,s-r)]^{2}\right)\gamma_{M+1-n}(r)dr,
\end{equation}
where $C_{1}(0,M,s):=1$ and $C_{2}(0,M,s):=0$, $s\geq 0$, $n\leq M$, and $M\in\mathbb{N}$. Then, we have
\begin{equation}
C_{1}(n,M,s)\leq\frac{1}{1-\rho}
\qquad\text{and}\qquad
C_{2}(n,M,s)\leq\frac{1}{(1-\rho)^{3}}.
\end{equation}
\end{lemma}

\begin{proof}
Let us use induction on $n$. For $n=0$, $C_{1}(0,M,s)=1\leq\frac{1}{1-\rho}$ since $\rho<1$.
Now assume $C_{1}(n-1,M,s)\leq\frac{1}{1-\rho}$, we get
\begin{align}
C_{1}(n,M,s)&\leq 1+\int_{0}^{s}\frac{1}{1-\rho}\gamma_{M+1-n}(r)dr
\\
&\leq 1+\frac{\Vert\gamma_{M+1-n}\Vert_{L^{1}}}{1-\rho}
\nonumber
\\
&\leq\frac{1}{1-\rho}.
\nonumber
\end{align}
It is clear that $C_{2}(0,M,s)=0\leq\frac{1}{(1-\rho)^{3}}$.
Now assume that $C_{2}(n-1,M,s)\leq\frac{1}{(1-\rho)^{3}}$ and apply the inequality
$C_{1}(n-1,M,s)\leq\frac{1}{1-\rho}$ that we have just proved, 
\begin{align}
C_{2}(n,M,s)
&\leq\int_{0}^{s}\left[\frac{1}{(1-\rho)^{3}}+\frac{1}{2}\frac{1}{(1-\rho)^{2}}\right]\gamma_{M+1-n}(r)dr
\\
&\leq\rho\left[\frac{1}{(1-\rho)^{3}}+\frac{1}{2}\frac{1}{(1-\rho)^{2}}\right]
\nonumber
\\
&\leq\frac{1}{(1-\rho)^{3}}.\nonumber
\end{align}
\end{proof}

\begin{lemma}\label{MDPLemma4}
Given any fixed $\theta\in\mathbb{R}$ and $a(t)$ as in Theorem \ref{MDPThm}, 
let $t$ be sufficiently large so that $k_{1}\frac{a(t)}{t}|\theta|\leq\frac{1-\rho}{4}$.
Then, we have
\begin{equation}\label{ErrorControl}
\left|f\left(n,M,\Gammaset,\frac{a(t)}{t}\theta,s\right)
-C_{1}(n,M,s)\frac{a(t)}{t}\theta
-C_{2}(n,M,s)\left(\frac{a(t)}{t}\theta\right)^{2}\right|
\leq k_{2}\left[\frac{a(t)}{t}|\theta|\right]^{3},
\end{equation}
where $C_{1}(n,M,s)$ and $C_{2}(n,M,s)$ are defined in Lemma \ref{MDPLemma3} and
\begin{equation}
k_{2}:=\frac{4\left[\frac{\rho}{2(1-\rho)^{3}}\left[k_{1}+\frac{1}{1-\rho}\right]+\rho k_{1}^{3}\right]}{(4-\rho)(1-\rho)}.
\end{equation}
\end{lemma}

\begin{proof}
Let us prove by induction. For $n=1$, $f\left(1,M,\Gammaset,\frac{a(t)}{t}\theta,s\right)=\frac{a(t)}{t}\theta$,
$C_{1}(0,M,s)=1$ and $C_{2}(0,M,s)=0$, thus \eqref{ErrorControl} holds.
Assume \eqref{ErrorControl} is true for $n-1$. Notice that
\begin{equation}
f\left(n,M,\Gammaset,\frac{a(t)}{t}\theta,s\right)
=\frac{a(t)}{t}\theta+\int_{0}^{s}\left[e^{f(n-1,M,\Gammaset,\frac{a(t)}{t}\theta,s-r)}-1\right]\gamma_{M+1-n}(r)dr.
\end{equation}
Since $|e^{x}-1-x-\frac{x^{2}}{2}|\leq|x|^{3}$ for $|x|<1$, and by Lemma \ref{MDPLemma2}, 
\begin{equation}
\left|f\left(n,M,\Gammaset,\frac{a(t)}{t}\theta,s\right)\right|
\leq k_{1}\frac{a(t)}{t}|\theta|
\leq\frac{1-\rho}{4}<1,
\end{equation}
we have
\begin{align}\label{boundI}
&\bigg|f\left(n,M,\Gammaset,\frac{a(t)}{t}\theta,s\right)
-\bigg[\frac{a(t)}{t}\theta
+\int_{0}^{s}\bigg[f\left(n-1,M,\Gammaset,\frac{a(t)}{t}\theta,s-r\right)
\\
&\qquad\qquad\qquad\qquad
+\frac{1}{2}\left(f\left(n-1,M,\Gammaset,\frac{a(t)}{t}\theta,s-r\right)\right)^{2}\bigg]
\gamma_{M+1-n}(r)dr\bigg]\bigg|
\nonumber
\\
&\leq\int_{0}^{s}\gamma_{M+1-n}(r)\left[k_{1}\frac{a(t)}{t}|\theta|\right]^{3}dr
\nonumber
\\
&\leq\rho k_{1}^{3}\left(\frac{a(t)|\theta|}{t}\right)^{3}.
\nonumber
\end{align}
In addition, by using the induction,
\begin{align}
L&:=\left|\left[f\left(n-1,M,\Gammaset,\frac{a(t)}{t}\theta,s-r\right)\right]^{2}
-C_{1}(n,M,s-r)^{2}\left(\frac{a(t)}{t}\theta\right)^{2}\right|
\\
&=\left|f\left(n-1,M,\Gammaset,\frac{a(t)}{t}\theta,s-r\right)
-C_{1}(n,M,s-r)\left(\frac{a(t)}{t}\theta\right)\right|
\nonumber
\\
&\qquad\qquad\qquad\cdot
\left|f\left(n-1,M,\Gammaset,\frac{a(t)}{t}\theta,s-r\right)
+C_{1}(n,M,s-r)\left(\frac{a(t)}{t}\theta\right)\right|
\nonumber
\\
&\leq\left[C_{2}(n,M,s-r)\left(\frac{a(t)}{t}\theta\right)^{2}
+k_{2}\left(\frac{a(t)}{t}|\theta|\right)^{3}\right]
\nonumber
\\
&\qquad\qquad\qquad\cdot
\left[\left|f\left(n-1,M,\Gammaset,\frac{a(t)}{t}\theta,s-r\right)\right|
+C_{1}(n,M,s-r)\frac{a(t)}{t}|\theta|\right].
\nonumber
\end{align}
Using the bounds in Lemma \ref{MDPLemma3}, we obtain
\begin{equation}
L\leq\left[\frac{1}{(1-\rho)^{3}}\left(\frac{a(t)}{t}\theta\right)^{2}
+k_{2}\left(\frac{a(t)}{t}|\theta|\right)^{3}\right]
\left[k_{1}\frac{a(t)}{t}|\theta|+\frac{1}{1-\rho}\frac{a(t)}{t}|\theta|\right].
\end{equation}
Since $k_{1}\frac{a(t)}{t}|\theta|\leq\frac{1-\rho}{4}$ and $k_{1}\geq\frac{1}{1-\rho}$, we get
\begin{align}\label{boundII}
L&\leq\left(\frac{a(t)}{t}|\theta|\right)^{3}\frac{1}{(1-\rho)^{3}}\left[k_{1}+\frac{1}{1-\rho}\right]
+k_{2}\left(\frac{a(t)}{t}|\theta|\right)^{3}
\left[k_{1}\frac{a(t)}{t}|\theta|+\frac{1}{1-\rho}\frac{a(t)}{t}|\theta|\right]
\\
&\leq\left(\frac{a(t)}{t}|\theta|\right)^{3}\frac{1}{(1-\rho)^{3}}\left[k_{1}+\frac{1}{1-\rho}\right]
+k_{2}\left(\frac{a(t)}{t}|\theta|\right)^{3}
\left[\frac{1-\rho}{4}+\frac{1}{1-\rho}\frac{1-\rho}{4k_{1}}\right]
\nonumber
\\
&\leq\left(\frac{a(t)}{t}|\theta|\right)^{3}
\left[\frac{1}{(1-\rho)^{3}}\left[k_{1}+\frac{1}{1-\rho}\right]+k_{2}\frac{1-\rho}{2}\right].
\nonumber
\end{align}
Let us put \eqref{boundI} and \eqref{boundII} together and define
\begin{align}
I&:=f\left(n,M,\Gammaset,\frac{a(t)}{t}\theta,s\right)-\frac{a(t)}{t}\theta
\left[1+\int_{0}^{s}C_{1}\left(n-1,M,s-r\right)\gamma_{M+1-n}(r)dr\right]
\\
&\qquad
-\left[\frac{a(t)}{t}\theta\right]^{2}
\bigg[\int_{0}^{s}\bigg[C_{2}\left(n-1,M,s-r\right)
\nonumber
\\
&\qquad\qquad\qquad
+\frac{1}{2}\left[C_{1}\left(n-1,M,s-r\right)\right]^{2}\bigg]
\gamma_{M+1-n}(r)dr\bigg].
\nonumber
\end{align}
Therefore,
\begin{align}
|I|&\leq\int_{0}^{s}k_{2}\left[\frac{a(t)}{t}|\theta|\right]^{3}\gamma_{M+1-n}(r)dr
\\
&\qquad
+\frac{1}{2}\int_{0}^{s}\left[\frac{1}{(1-\rho)^{3}}\left[k_{1}+\frac{1}{1-\rho}\right]+k_{2}\frac{1-\rho}{2}\right]
\left(\frac{a(t)}{t}|\theta|\right)^{3}\gamma_{M+1-n}(r)dr
\nonumber
\\
&\qquad\qquad\qquad
+\rho k_{1}^{3}\left(\frac{a(t)}{t}|\theta|\right)^{3}
\nonumber
\\
&\leq\left(\frac{a(t)}{t}|\theta|\right)^{3}
\left[\rho k_{2}+\frac{\rho}{2}\left[\frac{1}{(1-\rho)^{3}}\left[k_{1}+\frac{1}{1-\rho}\right]+k_{2}\frac{1-\rho}{2}\right]+
\rho k_{1}^{3}\right]
\nonumber
\\
&\leq\left(\frac{a(t)}{t}|\theta|\right)^{3}\left[\left[\frac{\rho(1-\rho)}{4}+\rho\right]k_{2}
+\frac{\rho}{2(1-\rho)^{3}}\left[k_{1}+\frac{1}{1-\rho}\right]+\rho k_{1}^{3}\right]
\nonumber
\\
&=k_{2}\left(\frac{a(t)}{t}|\theta|\right)^{3}.\nonumber
\end{align}
This proves the desired result.
\end{proof}

We observe that $k_{1}$ and $k_{2}$ only depend on $\rho$, so our bound in Lemma \ref{MDPLemma4}
is uniform in $n\leq M$, $M\in\mathbb{N}$ and $s\geq 0$. Now, let us go back
to the proof of Theorem \ref{MDPThm}.

\begin{proof}[Proof of Theorem \ref{MDPThm}]
We are interested to prove that the limit
\begin{equation}
\lim_{t\rightarrow\infty}\frac{t}{a(t)^{2}}\log
\mathbb{E}\left[e^{\frac{a(t)}{t}\theta(N_{t}-mt)}\right]
\end{equation}
exists and it can be computed explicitly.

Notice that
\begin{align}\label{ThreeI}
&\lim_{t\rightarrow\infty}\frac{t}{a(t)^{2}}
\log\mathbb{E}\left[e^{\frac{a(t)}{t}\theta(N_{t}-mt)}\right]
\\
&=\lim_{t\rightarrow\infty}\lim_{M\rightarrow\infty}
\frac{t}{a(t)^{2}}\left[\int_{0}^{t}(e^{f(M,M,\Gammaset,\theta,t-s)}-1)\gamma_{0}(s)ds-ma(t)\theta\right]
\nonumber
\\
&=\lim_{t\rightarrow\infty}\lim_{M\rightarrow\infty}\frac{t}{a(t)^{2}}
\bigg[\int_{0}^{t}\left(e^{C_{1}(M,M,t-r)\frac{a(t)}{t}\theta+C_{2}(M,M,t-r)(\frac{a(t)}{t}\theta)^{2}
+O(\frac{a(t)}{t}|\theta|)^{3}}-1\right)\gamma_{0}(r)dr
\nonumber
\\
&\qquad\qquad\qquad\qquad\qquad\qquad\qquad\qquad
-ma(t)\theta\bigg]
\nonumber
\\
&=\lim_{t\rightarrow\infty}\lim_{M\rightarrow\infty}
\frac{t}{a(t)^{2}}\left[\int_{0}^{t}C_{1}(M,M,t-r)\gamma_{0}(r)dr-mt\right]\frac{a(t)}{t}\theta
\nonumber
\\
&\qquad
+\lim_{t\rightarrow\infty}\lim_{M\rightarrow\infty}\frac{t}{a(t)^{2}}
\left(\int_{0}^{t}\left[C_{2}(M,M,t-r)+\frac{1}{2}C_{1}(M,M,t-r)^{2}\right]\gamma_{0}(r)dr\right)
\left(\frac{a(t)}{t}\theta\right)^{2}
\nonumber
\\
&\qquad\qquad
+\lim_{t\rightarrow\infty}O\left(\left(\frac{a(t)}{t}|\theta|\right)^{3}\right)
\int_{0}^{t}\gamma_{0}(r)dr\frac{t}{a(t)^{2}}
\nonumber
\\
&=:I_{1}+I_{2}+I_{3},
\nonumber
\end{align}
where we used the Taylor expansion of $e^{x}$ for $x=o(1)$.

The next step is to carry out careful analysis on $I_{1},I_{2},I_{3}$, the 
last three terms in \eqref{ThreeI}.

For the first term $I_{1}$ in \eqref{ThreeI}, it is easy to see that
\begin{align}
I_{1}&=\lim_{t\rightarrow\infty}\lim_{M\rightarrow\infty}
\frac{t}{a(t)}\theta\left[\frac{1}{t}\int_{0}^{t}C_{1}(M,M,t-r)\gamma_{0}(r)dr-m\right]
\\
&=\lim_{t\rightarrow\infty}
\frac{t}{a(t)}\theta\left[\frac{1}{t}\int_{0}^{t}C_{1}(\infty,\infty,t-r)\gamma_{0}(r)dr-m\right],
\nonumber
\end{align}
where $C_{1}(\infty,\infty,t)$ can be computed via iteration in \eqref{C1Defn} as
\begin{equation}
C_{1}(\infty,\infty,t)=1+\sum_{n=1}^{\infty}(\gamma_{n}\ast\cdots\ast\gamma_{1}\ast 1)(t).
\end{equation}
Let us recall that $m=\sum_{n=0}^{\infty}m_{n}$, where $m_{0}=\overline{\gamma}_{0}$ 
and $m_{n}=\overline{\gamma}_{0}\int_{0}^{\infty}(\gamma_{n}\ast\cdots\ast\gamma_{1})(s)ds$ for $n\geq1$.
Therefore, we have
\begin{align}
I_{1}&= \lim_{t\rightarrow\infty}\frac{t}{a(t)}\theta
\left[\frac{1}{t}\int_{0}^{t}\int_{0}^{t-r}h(s)ds\gamma_{0}(r)dr-m\right]
\\
&=\lim_{t\rightarrow\infty}\frac{t}{a(t)}\theta
\left[\frac{1}{t}\int_{0}^{t}h(s)ds\int_{0}^{t}\gamma_{0}(r)dr -\frac{1}{t} \int_{0}^{t}h(s)ds\int_{t-r}^{t}\gamma_{0}(r)dr-m\right]
\nonumber
\end{align}
 and 
\begin{align}\label{I1Three}
|I_{1}|&\leq|\theta|\limsup_{t\rightarrow\infty}\frac{t}{a(t)}
\left|\frac{1}{t}\int_{0}^{t}\gamma_{0}(s)ds-\overline{\gamma}_{0}\right|\int_{0}^{\infty}h(s)ds
\\
&\qquad\qquad
+|\theta|\limsup_{t\rightarrow\infty}\overline{\gamma}_{0}\frac{t}{a(t)}\int_{t}^{\infty}h(s)ds
\nonumber
\\
&\qquad\qquad\qquad\qquad
+|\theta|\limsup_{t\rightarrow\infty}\int_{0}^{\infty}h(s)
\left|\frac{1}{a(t)}\int_{t-s}^{t}\gamma_{0}(u)du\right|ds,
\nonumber
\end{align}
where $h(t):=1+\sum_{n=1}^{\infty}(\gamma_{n}\ast\cdots\ast\gamma_{1})(t)$
and thus $\int_{0}^{\infty}h(t)dt=\frac{1}{1-\rho}<\infty$. 
The first two terms in \eqref{I1Three} are zero due to Assumption \ref{AssumptionTwo}
and the third term in \eqref{I1Three} is zero by dominated convergence theorem.

Next, let us consider the second term $I_{2}$ in \eqref{ThreeI}. In Lemma \ref{MDPLemma3}, we obtained
a uniform bound on $C_{1}(M,M,s)$ and $C_{2}(M,M,s)$. 
Hence, we have
\begin{align}
I_{2}&=\lim_{t\rightarrow\infty}\lim_{M\rightarrow\infty}\frac{\theta^{2}}{t}
\int_{0}^{t}\left[C_{2}(M,M,t-r)+\frac{1}{2}C_{1}(M,M,t-r)^{2}\right]\gamma_{0}(r)dr
\\
&=\theta^{2}\lim_{t\rightarrow\infty}\frac{1}{t}\int_{0}^{\infty}\left[C_{2}(\infty,\infty,t-r)
+\frac{1}{2}C_{1}(\infty,\infty,t-r)^{2}\right]\gamma_{0}(r)dr
\nonumber
\\
&=\theta^{2}\lim_{t\rightarrow\infty}\left[C_{2}(\infty,\infty,t)
+\frac{1}{2}C_{1}(\infty,\infty,t)^{2}\right]\overline{\gamma}_{0}
\nonumber
\\
&=\frac{1}{2}\sigma^{2}\theta^{2},
\nonumber
\end{align}
since
\begin{equation}\label{Variance}
\overline{\gamma}_{0}\left[C_{2}(\infty,\infty,\infty)+\frac{1}{2}C_{1}(\infty,\infty,\infty)^{2}\right]
=\frac{1}{2}\sum_{j=0}^{\infty}\left(1+\sum_{p=1}^{\infty}\prod_{i=j+1}^{p+j}\int_{0}^{\infty}\gamma_{i}(u)du\right)^{2}m_{j}
=\frac{1}{2}\sigma^{2}.
\end{equation}

Let us verify \eqref{Variance}. First, fix $M$. We will let $M$ go to infinity later. 
We can do that since all
of our estimates for convergence in $M$ are uniform in $t$, so we can interchange the two limits.
Let us define $\Gammafunc(i,j,t)=(\gamma_{i}\ast\gamma_{i+1}\ast\cdots\ast\gamma_{j})(t)$ for $i\leq j$.
Another look at \eqref{C1Defn} reveals that
\begin{equation}
C_{1}(n,M,s)=1+[C_{1}(n-1,M,\cdot)\ast\gamma_{M+1-n}](s).
\end{equation}
Since $C_{1}(0,M,s)=1$, we get
\begin{equation}
C_{1}(n,M,s)=1+\sum_{k=M+1-n}^{M}(\Gammafunc(M+1-n,k,\cdot)\ast 1)(s),
\end{equation}
for $n\leq M$. Similarly,
\begin{equation}
C_{2}(n,M,s)=[C_{2}(n-1,M,\cdot)\ast\gamma_{M+1-n}](s)
+\frac{1}{2}[[C_{1}(n-1,M,\cdot)]^{2}\ast\gamma_{M+1-n}](s).
\end{equation}
Therefore,
\begin{equation}
C_{2}(n,M,s)=\frac{1}{2}\sum_{k=M+1-n}^{M}[(C_{1}(M-k,M,s))^{2}\ast\Gammafunc(M+1-k,k,\cdot)](s).
\end{equation}
Let us define
\begin{equation}
m(i,j):=\prod_{k=i+1}^{j}\int_{0}^{\infty}\gamma_{k}(u)du,
\end{equation}
for $j>i$ and $m(i,i):=1$. Hence, $m(0,j)=m_{j}$. Moreover, let us define
\begin{align}
I_{1}(M)&:=
\left[C_{2}(M,M,\infty)+\frac{1}{2}[C_{1}(M,M,\infty)]^{2}\right]\overline{\gamma}_{0}
\\
&=\frac{\overline{\gamma}_{0}}{2}\sum_{n=0}^{M}\left[1+\sum_{k=n+1}^{M}\int_{0}^{\infty}\Gammafunc(n+1,k,s)ds\right]^{2}m(0,n),
\nonumber
\end{align}
and
\begin{align}
I_{2}(M,t)&:=\frac{1}{t}\left[\left(C_{2}(M,M,\cdot)+\frac{1}{2}\left(C_{1}(M,M,\cdot)\right)^{2}\right)\ast\gamma_{0}\right](t)
\\
&=\frac{1}{2t}\left[\sum_{n=1}^{M}(C_{1}(M-n,M,\cdot))^{2}\ast\Gammafunc(1,n,\cdot)\ast\gamma_{0}
+(C_{1}(M,M,\cdot))^{2}\ast\gamma_{0}\right](t)
\nonumber
\end{align}
Now, 
\begin{align}
&I_{1}(M)-I_{2}(M,t)
\\
&=
\frac{1}{2}\sum_{n=0}^{M}\left[1+\sum_{k=n+1}^{M}(\Gammafunc(n+1,k,\cdot)\ast 1)(\infty)\right]^{2}
\left[\overline{\gamma}_{0}m(0,n)-\frac{(\Gammafunc(1,n,\cdot)\ast\gamma_{0})(t)}{t}\right]
\nonumber
\\
&\qquad
+\frac{1}{2t}\sum_{n=0}^{M}\left[\left[1+\sum_{k=n+1}^{M}(\Gammafunc(n+1,k,\cdot)\ast 1)(\infty)\right]^{2}
-\left[1+\sum_{k=n+1}^{M}(\Gammafunc(n+1,k,\cdot)\ast 1)(\cdot)\right]^{2}\right]
\nonumber
\\
&\qquad\qquad\qquad\qquad\ast\Gammafunc(1,n,\cdot)\ast\gamma_{0}(t).
\nonumber
\end{align}
We denote the first sum by $L_{1}$ and the second sum by $L_{2}$. 
Since $C_{1}(n,t,M)\leq\frac{1}{1-\rho}$ for all $n$ and $t$ by Lemma \ref{MDPLemma3},
\begin{align}
L_{1}&\leq\frac{1}{2(1-\rho)^{2}}\sum_{n=0}^{M}\overline{\gamma}_{0}m(0,n)-\frac{(\Gammafunc(1,n,\cdot)\ast\gamma_{0})(t)}{t}
\\
&\leq\frac{1}{2(1-\rho)^{2}}\left|\frac{1}{t}\int_{0}^{t}\gamma_{0}(s)ds-\overline{\gamma}_{0}\right|
\int_{0}^{\infty}h(s)ds
\nonumber
\\
&\qquad\qquad
+\overline{\gamma}_{0}\int_{t}^{\infty}h(s)ds
+\int_{0}^{\infty}h(s)\frac{1}{t}\int_{t-s}^{t}\gamma_{0}(u)duds.
\nonumber
\end{align}
The right hand side of the above equation goes to zero as $t\rightarrow\infty$ since $h(\cdot)$ is integrable.
Next, let us bound $L_{2}$.
\begin{align}
L_{2}&=\frac{1}{2t}\sum_{n=0}^{M}\left[\left[1+\sum_{k=n+1}^{M}(\Gammafunc(n+1,k,\cdot)\ast 1)(\infty)\right]
-\left[1+\sum_{k=n+1}^{M}(\Gammafunc(n+1,k,\cdot)\ast 1)(\cdot)\right]\right]
\\
&\qquad\cdot
\left[\left[1+\sum_{k=n+1}^{M}(\Gammafunc(n+1,k,\cdot)\ast 1)(\infty)\right]
+\left[1+\sum_{k=n+1}^{M}(\Gammafunc(n+1,k,\cdot)\ast 1)(\cdot)\right]\right]
\nonumber
\\
&\qquad\qquad\qquad\qquad\qquad\ast\Gammafunc(1,n,\cdot)\ast\gamma_{0}(t)
\nonumber
\\
&\leq\frac{1}{2t}\sum_{n=0}^{M}\frac{2}{1-\rho}
\left[\left[1+\sum_{k=n+1}^{M}(\Gammafunc(n+1,k,\cdot)\ast 1)(\infty)\right]
-\left[1+\sum_{k=n+1}^{M}(\Gammafunc(n+1,k,\cdot)\ast 1)(\cdot)\right]\right]
\nonumber
\\
&\qquad\qquad\qquad\qquad\qquad\ast\gamma_{0}(t)\cdot\rho^{n}
\nonumber
\\
&\leq\frac{1}{2t}\sum_{n=0}^{M}\frac{2\rho^{n}}{1-\rho}
\int_{0}^{t}\left[\int_{t-r}^{\infty}\sum_{k=n+1}^{M}{\Gammafunc(n+1,k,s)}ds\right]\gamma_{0}(r)dr
\nonumber
\\
&\leq\frac{1}{2t}\sum_{n=0}^{M}\frac{2\rho^{n}}{1-\rho}
\left[\int_{0}^{t}\sum_{k=n+1}^{M}{\Gammafunc(n+1,k,s)}\int_{t-r}^{t}\gamma_{0}(r)drds
+\int_{t}^{\infty}\sum_{k=n+1}^{M}{\Gammafunc(n+1,k,s)}\int_{0}^{t}\gamma_{0}(r)drds\right]
\nonumber
\\
&\leq\sum_{n=0}^{M}\frac{2\rho^{n}}{1-\rho}
\left[\int_{0}^{\infty}h(s)ds\frac{1}{t}\int_{t-r}^{t}\gamma_{0}(r)dr
+ \int_{t}^{\infty}h(s)ds(\overline{\gamma}_{0}+1)\right],
\nonumber
\end{align}
which goes to zero by dominated convergence theorem.
The difference in $M$ is given by
\begin{align}
&I_{1}(\infty)-I_{1}(M)
\\
&=\frac{\overline{\gamma}_{0}}{2}\sum_{n=M+1}^{\infty}
\left[1+\sum_{k=n+1}^{M}(\Gammafunc(n+1,k,\cdot)\ast 1)(\infty)\right]^{2}m(0,n)
\nonumber
\\
&\leq\frac{\overline{\gamma}_{0}}{2}\frac{1}{(1-\rho)^{2}}\sum_{n=M+1}^{\infty}\rho^{n},
\nonumber
\end{align}
and for sufficiently large $t$ (uniformly in $M$),
\begin{align}
&I_{2}(\infty,t)-I_{2}(M,t)
\\
&=\frac{1}{2t}\sum_{n=M+1}^{\infty}\left[1+\sum_{k=n+1}^{M}\int_{0}^{\cdot}\Gammafunc(n+1,k,s)ds\right]^{2}
\ast\Gammafunc(1,n,\cdot)\ast\gamma_{0}(t)
\nonumber
\\
&\leq\frac{1}{2(1-\rho)^{2}}\sum_{n=M+1}^{\infty}\rho^{n}\frac{1}{t}\int_{0}^{t}\gamma_{0}(r)dr
\nonumber
\\
&\leq\frac{\overline{\gamma}_{0}+1}{2(1-\rho)^{2}}\sum_{n=M+1}^{\infty}\rho^{n}.
\nonumber
\end{align}
Hence, we proved \eqref{Variance}.

Finally, let us show that the third term $I_{3}$ in \eqref{ThreeI} is zero in the limit.
For some universal constant $K>0$,
\begin{align}
|I_{3}|&\leq\limsup_{t\rightarrow\infty}K\left(\frac{a(t)}{t}|\theta|\right)^{3}
\left(\int_{0}^{t}\gamma_{0}(r)dr\right)\frac{t}{a(t)^{2}}
\\
&=\limsup_{t\rightarrow\infty}K\frac{a(t)}{t}|\theta|^{3}\frac{\int_{0}^{t}\gamma_{0}(r)dr}{t}
\nonumber
\\
&\leq\limsup_{t\rightarrow\infty}K\frac{a(t)}{t}|\theta|^{3}\overline{\gamma}_{0}\nonumber
\\
&=0.\nonumber
\end{align}

Hence, we proved that
\begin{equation}
\lim_{t\rightarrow\infty}\frac{t}{a(t)^{2}}\log
\mathbb{E}\left[e^{\frac{a(t)}{t}\theta(N_{t}-mt)}\right]=\frac{1}{2}\theta^{2}\sigma^{2}.
\end{equation}
By G\"{a}tner-Ellis theorem, the proof is complete.
\end{proof}

\begin{remark}
Following the same proof of Theorem \ref{MDPThm}, we can show that
for any $\theta\in\mathbb{R}$, 
$\lim_{t\rightarrow\infty}\mathbb{E}[e^{\frac{i\theta}{\sqrt{t}}(N_{t}-mt)}]
=e^{-\frac{\theta^{2}}{2}\sigma^{2}}$. 
In other words, our method gives an altenative proof to the central limit
theorem that was obtained in Fierro et al. \cite{Fierro}.
\end{remark}

\begin{remark}
Indeed, one can also consider the moderate deviations in the presence of random marks,
i.e. for a sequence of i.i.d. real-valued random variables $C_{1},C_{2},\ldots$
with mean $\mathbb{E}[C_{1}]$ and variance $\text{Var}[C_{1}]$ independent of $N_{t}$,
we expect that for a sequence $a(t)$ so that $\sqrt{t}\ll a(t)\ll t$,
$\mathbb{P}(\frac{\sum_{i=1}^{N_{t}}C{i}-\mathbb{E}[C_{1}]\mathbb{E}[N_{t}]}{a(t)}\in\cdot)$
follows a large deviation principle with rate function $J_{C}(x):=\frac{x^{2}}{2\sigma_{C}^{2}}$,
where $\sigma_{C}^{2}:=m\text{Var}[C_{1}]+\mathbb{E}[C_{1}]\sigma^{2}$
The proofs are similar to the proofs of moderate deviations for the unmarked case and we
will not go into the details in this paper.
\end{remark}

\section{Applications to Finance}\label{FinanceSection}

\subsection{Microstructure Noise}

Let $X_{t}$ stand for some asset price at time $t$. The signature plot can be defined
for $X_{t}$ over a time period $[0,T]$ at the time scale $\tau$ as
\begin{equation}\label{Vol}
\hat{C}(\tau):=\frac{1}{T}\sum_{n=0}^{\lfloor T/\tau\rfloor}(X_{(n+1)\tau}-X_{n\tau})^{2}.
\end{equation}
This is also known as the realized volatility. The microstructure noise effect
is described by an increase of the realized volatility when the time scale $\tau$ decreases.
This behavior is different from what one would expect if $X_{t}$ is a Brownian motion,
for which $\hat{C}(\tau)$ will be constant in $\tau$ as $T\rightarrow\infty$.

If $X^{1}_{t}$ and $X^{2}_{t}$ are the prices of two assets, we can define
\begin{equation}\label{Corr}
\hat{\rho}(\tau):=\frac{\hat{C}_{12}(\tau)}{\sqrt{\hat{C}_{1}(\tau)\hat{C}_{2}(\tau)}},
\end{equation}
where
\begin{equation}
\hat{C}_{12}(\tau):=\frac{1}{T}\sum_{n=0}^{\lfloor T/\tau\rfloor}
(X^{1}_{(n+1)\tau}-X^{1}_{n\tau})(X^{2}_{(n+1)\tau}-X^{2}_{n\tau}),
\end{equation}
and $\hat{C}_{1}(\tau)$ and $\hat{C}_{2}(\tau)$ are defined similarly as in \eqref{Vol}.

The Epps effect, named after Epps \cite{Epps} describes the 
pheonomenon that the correlation coefficient
$\hat{\rho}(\tau)$ increases in $\tau$ and it tends to zero as $\tau\rightarrow 0$.

Bacry et al. \cite{BacryII} studied the signature plot $\hat{C}(\tau)$ as in \eqref{Vol}
for the price model, $X_{t}=N_{1}(t)-N_{2}(t)$, where
$(N_{1},N_{2})$ is a bivariate Hawkes process and they also studied
correlation coefficient $\hat{\rho}(\tau)$ as in \eqref{Corr} for $X^{1}_{t}=N_{1}(t)-N_{2}(t)$,
$X^{2}_{t}=N_{3}(t)-N_{4}(t)$, where $(N_{1},N_{2},N_{3},N_{4})$ is a multivariate Hawkes process.
They considered the case of long horizon, i.e. the large $T$ limit and hence studied the macroscopic properties
of a multivariate Hawkes process, see e.g. \cite{Bacry}, \cite{BacryII}. The large $T$ limit can correspond to a trading day realization
of the price model. In \cite{BacryII}, they considered for instance a realization of $20$ hours (Figure 2 in \cite{BacryII}).

Following the ideas in \cite{Bacry}, \cite{BacryII}, 
one can do the same analysis for the Hawkes process with differnt exciting functions.
For example, we can fix a partition $(A_{1},A_{2})$ for $\mathbb{N}\cup\{0\}$ and 
let $N_{1}=\sum_{n\in A_{1}}N^{n}$ and $N_{2}=\sum_{n\in A_{2}}N^{n}$. 
Then, we can study the signature plot $\hat{C}(\tau)$ for $X_{t}=N_{1}(t)-N_{2}(t)$. 
One can also fix a partition $(A_{1},A_{2},A_{3},A_{4})$ for $\mathbb{N}\cup\{0\}$ and 
let $N_{i}=\sum_{n\in A_{i}}N^{n}$, $1\leq i\leq 4$.
Then, we can study the correlation coefficient $\hat{\rho}(\tau)$ for $X^{1}_{t}=N_{1}(t)-N_{2}(t)$,
$X^{2}_{t}=N_{3}(t)-N_{4}(t)$.

In the context of the Hawkes process with different exciting functions, since we already
proved ergodicity in Theorem \ref{ergodicity}, by considering large $T$, i.e. letting 
$T\rightarrow\infty$, by ergodic theorem,
\begin{equation}
\hat{C}(\tau)\rightarrow C(\tau):=\frac{1}{\tau}\mathbb{E}[(X_{\tau})^{2}],
\end{equation}
and
\begin{equation}
\hat{\rho}(\tau)\rightarrow\rho(\tau):=\frac{\mathbb{E}[X^{1}_{\tau}X^{2}_{\tau}]}
{\sqrt{\mathbb{E}[(X^{1}_{\tau})^{2}]\mathbb{E}[(X^{2}_{\tau})^{2}]}},
\end{equation}
as $T\rightarrow\infty$, where the expectations are taken over the stationary version
of the processes. Heuristically, as $\tau\rightarrow 0$, 
$\mathbb{E}[X^{1}_{\tau}X^{2}_{\tau}]=O(\tau^{2})$, 
$\mathbb{E}[(X^{1}_{\tau})^{2}]=O(\tau)$ and $\mathbb{E}[(X^{2}_{\tau})^{2}]=O(\tau)$.
Thus, as $\tau\rightarrow 0$, $\rho(\tau)=O(\tau)$ and this explains the vanishing
correlation coefficient as $\tau\rightarrow 0$ in the Epps effect.

Our main result is that $C(\tau)$ and $\rho(\tau)$ can be computed by evaluating $\mathbb{E}[(X^{1}_{\tau})^{2}]$, $\mathbb{E}[(X^{2}_{\tau})^{2}$,
and $\mathbb{E}[X^{1}_{\tau}X^{2}_{\tau}]$:

\begin{proposition}
Under Assumption \ref{MainAssumption},
\begin{align}
&\mathbb{E}[(X^{1}_{\tau})^{2}]
=\sum_{i\in A_{1}}\overline{\gamma}_{0}m_{i}\tau
+\sum_{i,j\in A_{1}}\int_{0}^{\tau}\int_{0}^{\tau}\rho(i,j,s-u)dsdu\label{X1Eqn}
\\
&\qquad+\sum_{i\in A_{2}}\overline{\gamma}_{0}m_{i}\tau
+\sum_{i,j\in A_{2}}\int_{0}^{\tau}\int_{0}^{\tau}\rho(i,j,s-u)dsdu
-2\sum_{i\in A_{1},j\in A_{2}}\int_{0}^{\tau}\int_{0}^{\tau}\rho(i,j,s-u)dsdu
\nonumber
\\
&\mathbb{E}[(X^{2}_{\tau})^{2}]
=\sum_{i\in A_{3}}\overline{\gamma}_{0}m_{i}\tau
+\sum_{i,j\in A_{3}}\int_{0}^{\tau}\int_{0}^{\tau}\rho(i,j,s-u)dsdu\label{X2Eqn}
\\
&\qquad+\sum_{n\in A_{4}}\overline{\gamma}_{0}m_{n}\tau
+\sum_{i,j\in A_{4}}\int_{0}^{\tau}\int_{0}^{\tau}\rho(i,j,s-u)dsdu
-2\sum_{i\in A_{3},j\in A_{4}}\int_{0}^{\tau}\int_{0}^{\tau}\rho(i,j,s-u)dsdu
\nonumber
\\
&\mathbb{E}[X^{1}_{\tau}X^{2}_{\tau}]
=\sum_{i\in A_{1},j\in A_{3}}\int_{0}^{\tau}\int_{0}^{\tau}\rho(i,j,s-u)dsdu
+\sum_{i\in A_{2},j\in A_{4}}\int_{0}^{\tau}\int_{0}^{\tau}\rho(i,j,s-u)dsdu\label{X1X2Eqn}
\\
&\qquad
-\sum_{i\in A_{2},j\in A_{3}}\int_{0}^{\tau}\int_{0}^{\tau}\rho(i,j,s-u)dsdu
-\sum_{i\in A_{1},j\in A_{4}}\int_{0}^{\tau}\int_{0}^{\tau}\rho(i,j,s-u)dsdu,
\nonumber
\end{align}
where $\rho(\cdot,\cdot,\cdot)$ are defined iteratively as $\rho(\cdot,\cdot,t)=\rho(\cdot,\cdot-t)$, $t>0$,
and for $t>s$, $i\geq 1$,
\begin{align}
\rho(i,i,t-s)&=\int_{-\infty}^{t}\int_{-\infty}^{s}\gamma_{i}(t-u)\gamma_{i}(s-v)\rho(i-1,i-1,|u-v|)dudv
\nonumber
\\
&\qquad\qquad
+\int_{-\infty}^{s}\gamma_{i}(t-u)\gamma_{i}(s-u)\overline{\gamma}_{0}m_{i-1}du,
\end{align}
and $\rho(0,0,t-s)=(\overline{\gamma}_{0})^{2}$, $t>s$ and for $j\geq i+1$, $t>s$,
\begin{equation}
\rho(i,j,t-s)=\int_{-\infty}^{s}\gamma_{j}(s-u)\rho(i,j-1,t-u)du,
\end{equation}
and finally, 
\begin{equation}
\rho(i,i+1,t-s)
=\begin{cases}
\int_{-\infty}^{s}\gamma_{i}(s-u)\rho(i,i,t-u)du &\text{if $t>s$}
\\
\int_{-\infty}^{s}\gamma_{i}(s-u)\rho(i,i,t-u)du+\gamma_{i}(s-t)\overline{\gamma}_{0}m_{i} &\text{if $t<s$}
\end{cases}.
\end{equation}
\end{proposition}

\begin{proof}
Let $N^{i}(dt):=N^{i}_{t+d\delta}-N^{i}_{t}$. 

First, for any $i\in\mathbb{N}\cup\{0\}$,
\begin{equation}
\frac{1}{d\delta}\mathbb{E}[N^{i}(dt)]=\overline{\gamma}_{0}m_{i},
\end{equation}
where $m_{i}$ is defined in \eqref{mn} for $i\in\mathbb{N}$ and $m_{0}:=1$.

Second, since $N^{i}$ is a simple point process,
\begin{equation}
\frac{1}{d\delta}\mathbb{E}[N^{i}(dt)N^{i}(dt)]=\frac{1}{d\delta}\mathbb{E}[N^{i}(dt)]=\overline{\gamma}_{0}m_{i}.
\end{equation}

Third, for any $t\neq s$, by stationarity, we can define
\begin{equation}
\rho(i,j,t-s):=\frac{1}{(d\delta)^{2}}\mathbb{E}[N^{i}(dt)N^{j}(ds)].
\end{equation}
Therefore, we can compute that
\begin{align}
\mathbb{E}[(X^{1}_{\tau})^{2}]
&=\mathbb{E}\left[\left(\sum_{n\in A_{1}}\int_{0}^{t}N^{n}(ds)-\sum_{n\in A_{2}}\int_{0}^{t}N^{n}(ds)\right)^{2}\right]
\\
&=\sum_{i\in A_{1}}\overline{\gamma}_{0}m_{i}\tau
+\sum_{i,j\in A_{1}}\int_{0}^{\tau}\int_{0}^{\tau}\rho(i,j,s-u)dsdu
\nonumber
\\
&\qquad+\sum_{i\in A_{2}}\overline{\gamma}_{0}m_{i}\tau
+\sum_{i,j\in A_{2}}\int_{0}^{\tau}\int_{0}^{\tau}\rho(i,j,s-u)dsdu
-2\sum_{i\in A_{1},j\in A_{2}}\int_{0}^{\tau}\int_{0}^{\tau}\rho(i,j,s-u)dsdu
\nonumber
\end{align}
Similarly, we can show \eqref{X2Eqn} and \eqref{X1X2Eqn}.

What remains is to compute $\rho(\cdot,\cdot,\cdot)$.
By symmetry,
\begin{equation}
\rho(i,i,t)=\rho(i,i,-t),\qquad-\infty<t<\infty.
\end{equation}
Therefore, for $t>s$, and $i\geq 1$,
\begin{align}
\rho(i,i,t-s)
&=\mathbb{E}[\lambda^{i}_{t}\lambda^{i}_{s}]
\\
&=\mathbb{E}\left[\int_{-\infty}^{t}\gamma_{i}(t-u)N^{i-1}(du)\int_{-\infty}^{s}\gamma_{i}(s-v)N^{i-1}(dv)\right]
\nonumber
\\
&=\int_{-\infty}^{t}\int_{-\infty}^{s}\gamma_{i}(t-u)\gamma_{i}(s-v)\rho(i-1,i-1,|u-v|)dudv
\nonumber
\\
&\qquad\qquad
+\int_{-\infty}^{s}\gamma_{i}(t-u)\gamma_{i}(s-u)\overline{\gamma}_{0}m_{i-1}du.
\nonumber
\end{align}
It is clear that $\rho(0,0,t-s)=(\overline{\gamma}_{0})^{2}$ for any $t>s$.

Fourth, for $j\geq i+1$,
\begin{align}
\rho(i,j,t-s)&=\frac{1}{(d\delta)^{2}}\mathbb{E}[N^{i}(dt)N^{j}(ds)]
\\
&=\frac{1}{d\delta}\mathbb{E}[N^{i}(dt)\lambda^{j}_{s}]
\nonumber
\\
&=\frac{1}{d\delta}\mathbb{E}\left[N^{i}(dt)\int_{-\infty}^{s}\gamma_{j}(s-u)N^{j-1}(du)\right]
\nonumber
\\
&=\int_{-\infty}^{s}\gamma_{j}(s-u)\rho(i,j-1,t-u)du.
\nonumber
\end{align}

Fifth and finally,
\begin{align}
\rho(i,i+1,t-s)&=\frac{1}{(d\delta)^{2}}\mathbb{E}[N^{i}(dt)N^{i+1}(ds)]
\\
&=\frac{1}{d\delta}\mathbb{E}[N^{i}(dt)\lambda^{i+1}_{s}]
\nonumber
\\
&=\frac{1}{d\delta}\mathbb{E}\left[N^{i}(dt)\int_{-\infty}^{s}\gamma_{i}(s-u)N^{i}(du)\right]
\nonumber
\\
&=
\begin{cases}
\int_{-\infty}^{s}\gamma_{i}(s-u)\rho(i,i,t-u)du &\text{if $t>s$}
\\
\int_{-\infty}^{s}\gamma_{i}(s-u)\rho(i,i,t-u)du+\gamma_{i}(s-t)\overline{\gamma}_{0}m_{i} &\text{if $t<s$}
\end{cases}.
\nonumber
\end{align}
\end{proof}

\subsection{Asymptotic Ruin Probabilities for a Risk Process with Hawkes Arrivals
with Different Exciting Functions}
In this section, we study the applications to ruin probabilities.
The applications of the Hawkes processes to ruin probabilities in insurnace have
been studied in Stabile and Torrisi \cite{Stabile}, Zhu \cite{ZhuRuin} for instance.
The advantage of using a Hawkes processes than a standard Poisson process is that the arrivals of the claims
will have a contagion and clustering effect.
We consider the following risk model for the surplus process $R_{t}$ of an insurance portfolio,
\begin{equation}
R_{t}=u+pt-\sum_{i=1}^{N_{t}}C_{i},
\end{equation}
where $u>0$ is the initial reserve, $p>0$ is the constant premium and 
the $C_{i}$'s are i.i.d. positive random variables
with $\mathbb{E}[e^{\theta C_{1}}]<\infty$ for any $\theta\in\mathbb{R}$. 
$C_{i}$ represents the claim size at the $i$th arrival time, 
these being independent of $N_{t}$, the Hawkes process with exciting functions
$(\gamma_{n})_{n\in\mathbb{N}\cup\{0\}}$.

For $u>0$, let
\begin{equation}
\tau_{u}=\inf\{t>0: R_{t}\leq 0\},
\end{equation}
and denote the infinite and finite horizon ruin probabilities by
\begin{equation}
\psi(u)=\mathbb{P}(\tau_{u}<\infty),\quad\psi(u,uz)=\mathbb{P}(\tau_{u}\leq uz),\quad u,z>0.
\end{equation}

We first consider the case when the claim sizes have light-tails, i.e. there exists 
some $\theta>0$ so that
$\mathbb{E}[e^{\theta C_{1}}]<\infty$.

By the law of large numbers,
\begin{equation}
\lim_{t\rightarrow\infty}\frac{1}{t}\sum_{i=1}^{N_{t}}C_{i}
=m\mathbb{E}[C_{1}].
\end{equation}

By Theorem \ref{CLDPThm}, 
$\Gamma_{C}(\theta):=\lim_{t\rightarrow\infty}\frac{1}{t}\log
\mathbb{E}[e^{\theta\sum_{i=1}^{N_{t}}C_{i}}]$ exists.  To exclude the trivial case, we assume that
\begin{equation}\label{between}
m\mathbb{E}[C_{1}]<p<\frac{\Gamma_{C}(\theta_{c})}{\theta_{c}},
\end{equation}
where the critical value $\theta_{c}$ is defined as
\begin{equation}
\theta_{c}:=\sup\{\theta:\Gamma_{C}(\theta)<\infty\}.
\end{equation}
The first inequality in \eqref{between} is the usual net profit condition in ruin theory
and the second inequality in \eqref{between} guarantees that
the equation $\Gamma_{C}(\theta)=p\theta$ has a unique positive 
solution $\theta^{\dagger}<\theta_{c}$. 

To see this, let $G(\theta)=\Gamma_{C}(\theta)-p\theta$. 
Notice that $G(0)=0$, $G(\infty)=\infty$, and that $G$ is convex. 
We also have $G'(0)=m\mathbb{E}[C_{1}]-p<0$ 
and $\Gamma_{C}(\theta_{c})-\rho\theta_{c}>0$ by \eqref{between}. 
Therefore, there exists only one
solution $\theta^{\dagger}\in(0,\theta_{c})$
of $\Gamma_{C}(\theta^{\dagger})=p\theta^{\dagger}$.

\begin{theorem}[Infinite Horizon]\label{InfiniteHorizon}
Assume \eqref{between}, we have
$\lim_{u\rightarrow\infty}\frac{1}{u}\log\psi(u)=-\theta^{\dagger}$, 
where $\theta^{\dagger}\in(0,\theta_{c})$ 
is the unique positive solution of $\Gamma_{C}(\theta)=p\theta$.
\end{theorem}

\begin{proof}
Let us first quote a result from Glynn and Whitt \cite{Glynn}.
Let $S_{n}$ be random variables and $\tau_{u}=\inf\{n: S_{n}>u\}$ 
and $\psi(u)=\mathbb{P}(\tau_{u}<\infty)$. 
Assume that there exist some $\gamma,\epsilon>0$ so that

(i) $\kappa_{n}(\theta)=\log\mathbb{E}[e^{\theta S_{n}}]$ 
is well defined and finite for $\gamma-\epsilon<\theta<\gamma+\epsilon$.

(ii) $\limsup_{n\rightarrow\infty}\mathbb{E}[e^{\theta(S_{n}-S_{n-1})}]<\infty$ 
for $-\epsilon<\theta<\epsilon$.

(iii) $\kappa(\theta)=\lim_{n\rightarrow\infty}\frac{1}{n}\kappa_{n}(\theta)$ 
exists and is finite for $\gamma-\epsilon<\theta<\gamma+\epsilon$.

(iv) $\kappa(\gamma)=0$ and $\kappa$ is differentiable at $\gamma$ with $0<\kappa'(\gamma)<\infty$.

Then, Glynn and Whitt \cite{Glynn} showed that 
$\lim_{u\rightarrow\infty}\frac{1}{u}\log\psi(u)=-\gamma$.

Take $S_{t}=\sum_{i=1}^{N_{t}}C_{i}-pt$ and 
$\kappa_{t}(\theta)=\log\mathbb{E}[e^{\theta S_{t}}]$. 
By Theorem \ref{CLDPThm}, we have 
$\lim_{t\rightarrow\infty}\frac{1}{t}\kappa_{t}(\theta)=\Gamma_{C}(\theta)-p\theta$. 
Consider $\{S_{nh}\}_{n\in\mathbb{N}}$. 
We have $\lim_{n\rightarrow\infty}\frac{1}{n}\kappa_{nh}(\theta)=h\Gamma_{C}(\theta)-hp\theta$. 
By checking the conditions (i)-(iv), we get 
\begin{equation}
\lim_{u\rightarrow\infty}\frac{1}{u}\log\mathbb{P}\left(\sup_{n\in\mathbb{N}}S_{nh}>u\right)
=-\theta^{\dagger}.
\end{equation}
Finally, notice that
\begin{equation}
\sup_{t\in\mathbb{R}^{+}}S_{t}\geq\sup_{n\in\mathbb{N}}S_{nh}\geq\sup_{t\in\mathbb{R}^{+}}S_{t}-ph.
\end{equation}
Hence, $\lim_{u\rightarrow\infty}\frac{1}{u}\log\psi(u)=-\theta^{\dagger}$.
\end{proof}

\begin{theorem}[Finite Horizon]\label{FiniteHorizon}
Under the same assumptions as in Theorem \ref{InfiniteHorizon}, we have
\begin{equation}
\lim_{u\rightarrow\infty}\frac{1}{u}\log\psi(u,uz)=-w(z),\quad\text{for any $z>0$},
\end{equation}
where
\begin{equation}
w(z)=
\begin{cases}
zI_{C}\left(\frac{1}{z}+p\right) &\text{if $0<z<\frac{1}{\Gamma'_{C}(\theta^{\dagger})-p}$}
\\
\theta^{\dagger} &\text{if $z\geq\frac{1}{\Gamma'_{C}(\theta^{\dagger})-p}$}
\end{cases}.
\end{equation}
\end{theorem}

\begin{proof}
The proof is similar to that in Stabile and Torrisi \cite{Stabile} and we omit it here.
\end{proof}

Next, we are interested to study the case when the claim sizes have heavy tails, 
i.e. $\mathbb{E}[e^{\theta C_{1}}]=+\infty$ for any $\theta>0$.

A distribution function $B$ is subexponential, i.e. $B\in\mathcal{S}$ if 
\begin{equation}
\lim_{x\rightarrow\infty}\frac{\mathbb{P}(C_{1}+C_{2}>x)}{\mathbb{P}(C_{1}>x)}=2,
\end{equation}
where $C_{1}$, $C_{2}$ are i.i.d. random variables with distribution function $B$. 
Let us denote $B(x):=\mathbb{P}(C_{1}\geq x)$
and let us assume that $\mathbb{E}[C_{1}]<\infty$ and define 
$B_{0}(x):=\frac{1}{\mathbb{E}[C]}\int_{0}^{x}\overline{B}(y)dy$,
where $\overline{F}(x)=1-F(x)$ is the complement of any distribution function $F(x)$.
The examples and properties of subexponential distributions can be found in the book by Asmussen and Albrecher \cite{Asmussen}.

Goldie and Resnick \cite{Goldie} showed that if $B\in\mathcal{S}$ and satisfies some smoothness
conditions, then $B$ belongs to the maximum domain of attraction of either the Frechet distribution
or the Gumbel distribution. In the former case, $\overline{B}$ is regularly varying,
i.e. $\overline{B}(x)=L(x)/x^{\alpha+1}$, for some $\alpha>0$ and we write
it as $\overline{B}\in\mathcal{R}(-\alpha-1)$, $\alpha>0$.

We assume that $B_{0}\in\mathcal{S}$ and either $\overline{B}\in\mathcal{R}(-\alpha-1)$ 
or $B\in\mathcal{G}$, i.e.
the maximum domain of attraction of Gumbel distribution. $\mathcal{G}$ 
includes Weibull and lognormal distributions.

When the arrival process $N_{t}$ satisfies a large deviation result, 
the probability that it deviates away
from its mean is exponentially small, which is dominated by subexonential distributions. 
The results in Zhu \cite{ZhuRuin}
for the asymptotics of ruin probabilities for risk processes with non-stationary, 
non-renewal arrivals and subexponential claims can be applied in the context of Hawkes arrivals
with different exciting functions.
We have the following infinite-horizon and finite-horizon ruin probability estimates 
when the claim sizes are subexponential.

\begin{theorem}
Assume the net profit condition $p>m\mathbb{E}[C_{1}]$.

(i) (Infinite-Horizon)
\begin{equation}
\lim_{u\rightarrow\infty}\frac{\psi(u)}{\overline{B}_{0}(u)}
=\frac{m\mathbb{E}[C_{1}]}{p-m\mathbb{E}[C_{1}]}.
\end{equation}

(ii) (Finite-Horizon) For any $T>0$,
\begin{equation}
\lim_{u\rightarrow\infty}\frac{\psi(u,uz)}{\overline{B}_{0}(u)}
=
\begin{cases}
\frac{m\mathbb{E}[C_{1}]}{p-m\mathbb{E}[C_{1}]}
\left[1-\left(1+\left(1-\frac{m\mathbb{E}[C_{1}]}{p}\right)
\frac{T}{\alpha}\right)^{-\alpha}\right]
&\text{if $\overline{B}\in\mathcal{R}(-\alpha-1)$}
\\
\frac{m\mathbb{E}[C_{1}]}{p-m\mathbb{E}[C_{1}]}
\left[1-e^{-(1-\frac{m\mathbb{E}[C_{1}]}{p})T}\right]&\text{if $B\in\mathcal{G}$}
\end{cases}.
\end{equation}
\end{theorem}

\section*{Acknowledgements}

The authors are extremely grateful to the editor and the referees for a very careful reading of the manuscript
and also for the very helpful suggestions.

\end{document}